\documentclass[12 pt, psamsfonts]{amsart}
\usepackage{amsmath}
\usepackage{amsthm}
\usepackage{amsfonts}
\usepackage{amssymb}
\usepackage{float}
\usepackage{pgf,tikz}
\usepackage{mathrsfs}
\usetikzlibrary{arrows}
\usepackage{eucal}
\usepackage{graphicx}
\usepackage{multirow}
\usepackage[all,knot]{xy}
\xyoption{arc}

\newcommand{\CC}{\mathcal{C}}
\newcommand{\LL}{\mathcal{L}}
\newcommand{\QQ}{\mathcal{Q}}

\newcommand{\End}{{\rm End}}
\newcommand{\Rep}{{\rm Rep}}

\newcommand{\ot}{\otimes}
\newcommand{\B}{\mathcal{B}}

\newcommand{\iv}{^{-1}}
\newcommand{\unit}{\mathbf{1}}

\newcommand{\Hom}{{\rm Hom}}

\newcommand{\one}{{\rm Id}}

\newcommand{\Z}{\mathbb{Z}}

\newcommand{\Q}{\mathbb{Q}}
\newcommand{\C}{\mathbb{C}}
\newcommand{\R}{\mathbb{R}}

\newcommand{\Aut}{{\rm Aut}}
\newcommand{\ppm}[1]{{\textcolor{green}{#1}}}
\newcommand{\er}[1]{{\textcolor{blue}{#1}}}
\newcommand{\amk}[1]{{\textcolor{purple}{#1}}}
\newcommand{\NB}{\mathcal{NB}}  %% just a hook for necklace braid gp name
\newcommand{\LB}{\mathcal{LB}}
\newcommand{\CB}{\mathcal{CB}} %% loop braid
\newcommand{\ignore}[1]{}  %% simply ignores the contents!

\newcommand{\beq}{\begin{equation}}
\newcommand{\eq}{\end{equation}}

\newcommand{\mat}[1]{\left( \begin{array}{#1} }
\newcommand{\tam}{\end{array} \right)}

\newtheorem{theorem}{Theorem}[section]
\newtheorem*{theorem*}{Theorem}
\newtheorem{lemma}[theorem]{Lemma}
\newtheorem{prop}[theorem]{Proposition}
\theoremstyle{definition}

\newtheorem{example}[theorem]{Example}

\newtheorem{question}[theorem]{Question}
\newtheorem{conj}[theorem]{Conjecture}
\newtheorem{cor}[theorem]{Corollary}
\theoremstyle{remark}
\newtheorem{remark}[theorem]{Remark}
\newtheorem{remarks}[theorem]{Remarks}

\numberwithin{equation}{section}

\def\hrad{0.5}
\def\vrad{0.1}
\def\rsep{0.25}
\def\miny{0}
\def\maxy{4}
\def\tension{0.5}
\calclayout

\begin{document}

\title[Representations of the necklace braid group]{Representations of the necklace braid group: topological and combinatorial approaches}
\author{Alex Bullivant$^{1}$, Andrew Kimball$^2$, Paul Martin$^1$, Eric C. Rowell$^2$}

\address{$^1$Department of Pure Mathematics\\
University of Leeds\\Leeds, LS2 9JT\\UK}
\email{A.L.Bullivant@leeds.ac.uk, ppmartin@maths.leeds.ac.uk }
\address{$^2$Department of Mathematics\\
    Texas A\&M University\\
    College Station, TX 77843-3368\\
    U.S.A.}
\email{amkimball1@math.tamu.edu, rowell@math.tamu.edu}

\thanks{ER and AK gratefully acknowledge support under USA NSF grant DMS-1664359, and AB and PM thank EPSRC for support under Grant EP/I038683/1. PM and ER thank Celeste Damiani for useful conversations. }


\begin{abstract}
The necklace braid group $\NB_n$ is the motion group of the $n+1$ component necklace link $\LL_n$ in Euclidean $\R^3$.  Here $\LL_n$ consists of $n$ pairwise unlinked Euclidean circles each linked to an auxiliary circle.
 Partially motivated by physical considerations, we study representations of the necklace braid group $\NB_n$, especially those obtained as extensions of representations of the braid group $\B_n$ and the loop braid group $\LB_n$.  We show that any irreducible $\B_n$ representation extends to $\NB_n$ in a standard way.  We also find some non-standard extensions of several well-known $\B_n$-representations such as the Burau and LKB representations. Moreover, we prove that any local representation of $\B_n$ (i.e. coming from a braided vector space) can be extended to $\NB_n$, in contrast to the situation with $\LB_n$.  We also discuss some directions for future study from categorical and physical perspectives.
\end{abstract}

\maketitle

\section{Introduction}
Topology, like many fields of mathematics, owes some of its early development to questions arising in physics.  A classic example of this is the development of knot theory: Lord Kelvin and Tait, inspired by experiments of Helmholtz, theorized that atoms were knotted tubes of \ae ther, distinguished by their knot type \cite{Silver}.  This theory was quickly dismissed, but Tait's tabulation of knot projections with few crossings is arguably the dawn of modern knot theory.  

Non-abelian statistics of anyons in two spatial dimensions has attracted considerable attention largely due to topological quantum computation \cite{nayaktqc,wang10}.  Exchanging non-abelian anyons induces unitary representations of the braid group $\B_n$, which can yield braiding-only universal quantum computation models.  Mathematically, this is a rich theory because the braid group acts faithfully on the fundamental group of the punctured plane.  The well-studied framework of $(2+1)$-TQFTs can be used to systematically study these representations and their vast generalization to mapping class groups of punctured surfaces of any genus.

Naturally we would like to extend these ideas to $3$-dimensional topological materials.  Unfortunately $(3+1)$-TQFTs are not as well studied so we cannot obtain as explicit descriptions as in the $2$-dimensional case.  
Instead, here we will study the relevant 
{\em motion groups} \cite{Gold}
and their representations from a more elementary algebraic perspective.

 An extension of non-abelian statistics of point-like excitations to three spatial dimensions is not possible due to the spin-statistics theorem in its exchange statistics formulation: exchanging the positions of two indistinguishable particles changes their state vector by at most a sign.  Mathematically, the motion group of $n$ identical points in $\R^3$ is the symmetric group $\mathfrak{S}_n$ which leads to the possibility of parastatistics.  Notice also that the fundamental group of $\R^3$ with $n$ points deleted is trivial, so that motions of points cannot be detected in this way -- 
 one must consider the framing of paths to explain the spin-statistics theorem topologically.  But in any case points in $3$-dimensions are much less interesting than in $2$-dimensions.  
 
 On the other hand, 
 {\em loop} or closed string excitations occur naturally in condensed matter physics and string theory. 
 The mathematical manifestation of this idea
 was considered in \cite{kmrw,bruetal} with a study of local and low dimension representations of the loop braid group $\LB_n$: the group of motions of $n$ oriented circles in $\R^3$.  
 In this article we consider the related group 
$\NB_n$ 
of motions (up to isotopy)
of a necklace $\mathcal{L}_n$: $n$ unlinked oriented circles that are linked to another auxiliary oriented circle, see Figure \ref{fig1}.  One compelling reason to undertake this study is that such a configuration may be more feasible physically than the free loop picture.  Indeed, a number of proposals in this direction have appeared recently, see \cite{WL,CTW,LinLevin,JMR}.  
 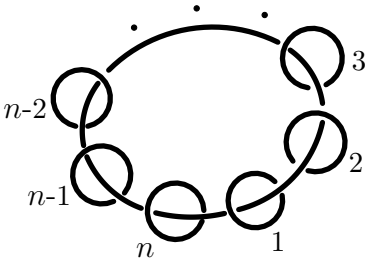
\begin{figure}[h!] %For now, it's something, but I will work on resizing and recentering it. I made it in Geogebra and converted it to a tikz file.
\begin{tikzpicture}[line cap=round,line join=round,>=triangle 45,x=1.0cm,y=1.0cm,scale=.4]
\clip(3.8,5.3) rectangle (17.22,15.04);
\draw [shift={(7.143409691629958,8.741409691629958)},line width=2.pt]  plot[domain=-0.6715805404685691:5.135259531759732,variable=\t]({1.*0.9665484155614196*cos(\t r)+0.*0.9665484155614196*sin(\t r)},{0.*0.9665484155614196*cos(\t r)+1.*0.9665484155614196*sin(\t r)});
\draw [shift={(9.600193548387093,7.531634408602146)},line width=2.pt]  plot[domain=0.030173215557652316:5.92271473679269,variable=\t]({1.*0.9402344246349372*cos(\t r)+0.*0.9402344246349372*sin(\t r)},{0.*0.9402344246349372*cos(\t r)+1.*0.9402344246349372*sin(\t r)});
\draw [shift={(6.48,11.290319148936172)},line width=2.pt]  plot[domain=-1.3590392941302882:4.500631947720082,variable=\t]({1.*0.9515743370211937*cos(\t r)+0.*0.9515743370211937*sin(\t r)},{0.*0.9515743370211937*cos(\t r)+1.*0.9515743370211937*sin(\t r)});
\draw [shift={(12.250573813249874,7.8885602503912375)},line width=2.pt]  plot[domain=-5.511828625039063:0.19041787825014977,variable=\t]({1.*0.9057982829652452*cos(\t r)+0.*0.9057982829652452*sin(\t r)},{0.*0.9057982829652452*cos(\t r)+1.*0.9057982829652452*sin(\t r)});
\draw [shift={(14.252784810126585,9.825601265822785)},line width=2.pt]  plot[domain=-1.8320151706568284:3.834981442045294,variable=\t]({1.*0.9788063721474181*cos(\t r)+0.*0.9788063721474181*sin(\t r)},{0.*0.9788063721474181*cos(\t r)+1.*0.9788063721474181*sin(\t r)});
\draw [shift={(14.144442413162706,12.623345521023767)},line width=2.pt]  plot[domain=-1.094653172979898:4.566543181398944,variable=\t]({1.*0.9938973913024262*cos(\t r)+0.*0.9938973913024262*sin(\t r)},{0.*0.9938973913024262*cos(\t r)+1.*0.9938973913024262*sin(\t r)});
\draw [shift={(9.12252737006432,10.210945147042423)},line width=2.pt]  plot[domain=2.4985387663157033:3.3309167550133902,variable=\t]({1.*2.6262574945073136*cos(\t r)+0.*2.6262574945073136*sin(\t r)},{0.*2.6262574945073136*cos(\t r)+1.*2.6262574945073136*sin(\t r)});
\draw [shift={(9.549538420223378,10.59818426683103)},line width=2.pt]  plot[domain=3.549895026455385:4.359443518328843,variable=\t]({1.*3.1611172091857687*cos(\t r)+0.*3.1611172091857687*sin(\t r)},{0.*3.1611172091857687*cos(\t r)+1.*3.1611172091857687*sin(\t r)});
\draw [shift={(10.122447072586986,12.345681792593489)},line width=2.pt]  plot[domain=4.473986800724279:4.934235135071265,variable=\t]({1.*5.009811604272178*cos(\t r)+0.*5.009811604272178*sin(\t r)},{0.*5.009811604272178*cos(\t r)+1.*5.009811604272178*sin(\t r)});
\draw [shift={(10.730264302656016,11.302785479095485)},line width=2.pt]  plot[domain=4.967087241105105:6.06853738985942,variable=\t]({1.*3.8388867891986287*cos(\t r)+0.*3.8388867891986287*sin(\t r)},{0.*3.8388867891986287*cos(\t r)+1.*3.8388867891986287*sin(\t r)});
\draw [shift={(11.87801177868133,10.750642452885886)},line width=2.pt] plot[domain=0.14099215481533797:0.9500949355865034,variable=\t]({1.*2.631204653369704*cos(\t r)+0.*2.631204653369704*sin(\t r)},{0.*2.631204653369704*cos(\t r)+1.*2.631204653369704*sin(\t r)});
\draw [shift={(10.836581200432741,8.76693850897955)},line width=2.pt]  plot[domain=1.113237616533615:2.3565726019025384,variable=\t]({1.*4.892498560257611*cos(\t r)+0.*4.892498560257611*sin(\t r)},{0.*4.892498560257611*cos(\t r)+1.*4.892498560257611*sin(\t r)});
\begin{scriptsize}
\draw [fill=black] (8.22,13.64) circle (2.5pt);
\draw [fill=black] (10.3,14.28) circle (2.5pt);
\draw [fill=black] (12.56,14.06) circle (2.5pt);
\end{scriptsize}
\node (a1) at (8.6,5.3) [label=$n$]{};
\node (a2) at (5.4,7) [label=$n$-{\small 1}]{};
\node (a3) at (4.6,9.9) [label=$n$-{\small 2}]{};
\node (a4) at (13,5.5) [label={\small 1}]{};
\node (a5) at (15.6,8.1) [label={\small 2}]{};
\node (a6) at (15.7,11.5) [label={\small 3}]{};
\end{tikzpicture}
  \caption{The necklace $\LL_n$
  as seen from a generic observation point}
 \end{figure}\label{fig1}

 The motion group $\NB_n$,
 the  \emph{necklace braid group}, 
 is described in \cite{BB16}, where it is
 identified with the fundamental group of the configuration space of
 $\mathcal{L}_n$.
 In $\mathcal{L}_n$
 we fix a circle labelled $1$, and 
 order the $n$ circles $1,\ldots,n$ 
 in a counterclockwise fashion;
 and orient the auxiliary circle in the same counterclockwise way.
 The % \emph{necklace braid group} 
 group $\NB_n$ 
 includes elements $\sigma_1,\ldots,\sigma_n,\tau$ 
 where
 % Geometrically 
  $\sigma_i$ is the motion,
  up to homotopy, of passing the $i$th circle through the $i+1$st, while $\tau$ corresponds to shifting each circle one position in the counterclockwise direction.  We use the function convention when composing elements of the motion group: $fg$ means apply $g$ then $f$.
In fact $\NB_n$  is generated by these elements. 

\medskip

Although the realisation of $\NB_n$ as a motion group is fundamental to its utility in physical modelling, 
from a representation theory perspective, manipulating  $\NB_n$ at the `geometric topological' level of its definition as a motion group is relatively hard. Fortunately \cite[Theorem 2.3]{BB16} gives a presentation by abstract generators and relations that facilitates such manipulations. Next we discuss this `combinatorial' realisation. 

%Furthermore we have the following. 

% \ppm{...maybe say what these are first as elements of $\NB_n$ (as below); then give the 
% presentation as a Theorem? (If that's ok I will implement.) }\er{I agree, thanks!}
 
\begin{theorem}[\cite{BB16}]\label{thm:defining relations}
We have a presentation of a group isomorphic to $\NB_n$ 
%is isomorphic  (under the obvious implied map) to the abstract group with
by abstract generators 
 $\sigma_1, \ldots, \sigma_n$, $\tau$
 %using the same symbols to indicate how the isomorphism works
satisfying: 
\begin{enumerate}
  \item[(B1)] $\sigma_i\sigma_{i+1}\sigma_i=\sigma_{i+1}\sigma_i\sigma_{i+1}$ 
 \item[(B2)] $\sigma_i\sigma_j=\sigma_j\sigma_i$ for $|i-j|\neq 1\pmod{n}$,
 \item[(N1)] $\tau\sigma_i\tau^{-1}=\sigma_{i+1}$ for $1\leq i\leq n$ 
 \item[(N2)] $\tau^{2n}=1$
 \end{enumerate}
 Here indices are taken modulo $n$,  with $\sigma_{n+1}:=\sigma_1$ and $\sigma_0:=\sigma_n$.
 \qed
 \end{theorem}

\newcommand{\BA}{{\mathcal B}\tilde{A}} % affine braid 

Observe from the presentation 
in Theorem \ref{thm:defining relations} that 
the subgroup generated by the $\sigma_i$ for $1\leq i\leq n-1$ is a quotient of Artin's braid group $\B_n$.  It is not hard to verify that (N1) and (N2) do not induce further relations among $\sigma_1,\ldots,\sigma_{n-1}$, so that, in fact, we have $\B_n<\NB_n$.  

It also follows from the presentation that $\NB_n$ contains a normal subgroup isomorphic to the \textit{affine braid group} (of type $A$) on $n$ strands: 
$
%B\tilde{A}_{n}
\BA_n
\; \cong \; 
\langle\sigma_1,\ldots,\sigma_n\rangle
\; \lhd \;  \NB_n$.  
Notice that 
$[\NB_n:\BA_{n}]=2n$ and 
$\NB_n=\BA_n\rtimes \langle \tau\rangle$. 
In particular, every element of $\NB_n$ may be written as $\tau^k\beta$ with $\beta\in \BA_n$. 

%It is also clear \cite{BB16} that 
%$\NB_n$ is isomorphic to a quotient of $\CB_n$. 

The \emph{annular} or \emph{circular braid group} $\CB_n$ 
(the fundamental group of the configuration space of $n$ points in an annulus) may be presented
by generators as for $\NB_n$ 
in Theorem~\ref{thm:defining relations}
but omitting relation (N2) (see e.g. \cite{BB16}).
Thus 
$\NB_n$ is isomorphic to a quotient of $\CB_n$. 
\begin{equation} \label{eq:C2N}
\zeta: \CB_n \rightarrow \NB_n
\end{equation}
%\ppm{I've clumsily attached a label here 
%as a reminder to consider tidying up a loose end - since this bit is related to \S\ref{ss:B2N}...}

\medskip

Some of the  relations 
(B1-N2)
for $\NB_n$ are redundant: the following reduces the number of defining relations from $\frac{1}{2}n(n+1)+1$  to $2n-1$.
\begin{lemma}\label{relationslemma}
The relations  (N2), (N1), (B1) for $i=1$ (i.e. $\sigma_1\sigma_2\sigma_1=\sigma_2\sigma_1\sigma_2$), and (B2) for $i=1$ and $3\leq j\leq n-1$ (i.e. $\sigma_1\sigma_j=\sigma_j\sigma_1$ for $3\leq j\leq n-1$), imply all relations of Theorem \ref{thm:defining relations}.
\end{lemma}
\begin{proof}
Assuming (N1) gives us $\tau^{i-1}\sigma_1\tau^{-i+1}=\sigma_i$ for all $i$ where indices are taken modulo $n$. Thus $\sigma_1\sigma_2\sigma_1=\sigma_2\sigma_1\sigma_2$ implies that for any $i$:
\begin{eqnarray*}
\sigma_i\sigma_{i+1}\sigma_i & = & (\tau^{i-1}\sigma_1\tau^{-i+1})(\tau^{i}\sigma_1\tau^{-i})(\tau^{i-1}\sigma_1\tau^{-i+1})\\ & = & \tau^{i-1}\sigma_1\tau^{1}\sigma_1\tau^{-1}\sigma_1\tau^{-i+1}\\
  & = & \tau^{i-1}\sigma_1\sigma_2\sigma_1\tau^{-i+1}\\ & = & \tau^{i-1}\sigma_2\sigma_1\sigma_2\tau^{-i+1}\\
& = & \tau^{i}\sigma_1\tau^{-1}\sigma_1\tau^{1}\sigma_1\tau^{-i}\\ & = & (\tau^{i}\sigma_1\tau^{-i})(\tau^{i-1}\sigma_1\tau^{-i+1})(\tau^{i}\sigma_1\tau^{-i})\\
& = & \sigma_{i+1}\sigma_i\sigma_{i+1}. 
\end{eqnarray*}
\iffalse
\er{I think this next step is superfluous: we assume (N1) holds mod $n$ so the above works for $i=n$ (identifying $\sigma_{n+1}=\sigma_1$.} Next using $\sigma_1=\tau^n\sigma_1\tau^{-n}$ and $\sigma_n=\tau^{n-1}\sigma_1\tau^{-n+1}$ we get:
\begin{eqnarray*}
\sigma_1\sigma_n\sigma_1 & = & (\tau^n\sigma_1\tau^{-n})(\tau^{n-1}\sigma_1\tau^{-n+1})(\tau^n\sigma_1\tau^{-n})\\
 & = & \tau^n\sigma_1\tau^{-1}\sigma_1\tau\sigma_1\tau^{-n}\\
 & = & \tau^{n-1}\sigma_2\sigma_1\sigma_2\tau^{-n+1}\\
 & = & \tau^{n-1}\sigma_1\sigma_2\sigma_1\tau^{-n+1}\\
 & = & (\tau^{n-1}\sigma_1\tau^{-n+1})(\tau^{n-1}\sigma_2\tau^{-n+1})(\tau^{n-1}\sigma_1\tau^{-n+1})\\
 & = & \sigma_n\sigma_1\sigma_n
 \end{eqnarray*}
 \fi

 Next we verify (B2) assuming $\sigma_1$ commutes with $\sigma_k$ with $3\leq k\leq n-1$.
We may assume $n\geq j>i>1$ and $|j-i|\not\equiv1\pmod{n}$.
\begin{eqnarray*}
  \sigma_i\sigma_j & = & \tau^{i-1}\sigma_1\tau^{-i+1}\tau^{i-1}\sigma_{j-i+1}\tau^{-i+1}\\
 & = & \tau^{i-1}\sigma_1\sigma_{j-i+1}\tau^{-i+1}\\ & = & \tau^{i-1}\sigma_{j-i+1}\sigma_1\tau^{-i+1}\\
 & = & \tau^{i-1}\sigma_{j-i+1}\tau^{-i+1}\tau^{i-1}\sigma_1\tau^{-i+1}  = \sigma_j\sigma_i
\end{eqnarray*}
\end{proof}

Another easy observation is that $\sigma_i\mapsto (i\;i+1)$ (modulo $n$, so that $\sigma_n\mapsto(n\;1)$) and $\tau\rightarrow (1\;2\;\cdots\;n)$ gives a surjection $\varphi: \NB_n\rightarrow S_n$.  The kernel of $\varphi$ is the normal subgroup of motions that carry each circle back to their original positions: the \emph{pure necklace braids}.

The motion group of $n$ unlinked oriented circles in $\R^3$, i.e. the Loop braid group $\LB_n$ studied in \cite{BaezCransWise} does not contain the necklace braid group, but it does contain a quotient of $\NB_n$ by a central subgroup of order 2 as we will see.  This relationship is explored in section 2.

\subsection{Overview of paper}

In \S\ref{ss:fromB} we lay down some basic facts about the relationships between $\NB_n$ and various other topologically constructed groups.  In particular we develop a number of ways to construct representations of $\NB_n$ 
both by extending from $\B_n$ representations or by factoring through representations of the loop braid group.  These ideas are applied in \S\ref{s:extensions} to many well-known examples of $\B_n$ representations, while in \S\ref{ss:LR} we focus on extending \emph{local} $\B_n$ representations.  It turns out to be much easier to construct local representations of $\NB_n$ than $\LB_n$.  
%\ppm{IS THIS OK?!:
In \S\ref{ss:low-d} we address the `brute' algebraic 
analysis of 
low-dimensional representations
(in the spirit of \cite{bruetal,LRagt,Formanek}).
In \S\ref{ss:BFC} we 
return closer to physical considerations. We
discuss the construction of representations from braided fusion categories; and  
in \S\ref{ss:XXZ} we explore more direct physical manifestations of $\NB_n$ using `categorified' quantum spin chains.
%}

%%%%%%%%%%%%%%%%%%%%%%%%%%%%%%%%%%%%%%%%%%%%%%%%%%%%%%%%%%%%%%%%%%%%%%%%%%%%%
%%%%%%%%%%%%%%%%%%%%%%%%%%%%%%%%%%%%%%%%%%%%%%%%%%%%%%%%%%%%%%%%%%%%%%%%%%%%%
\section{$\NB_n$ representations from $\B_n$ and $\LB_n$ representations}
\label{ss:fromB}
The necklace braid group is closely related to both the ordinary braid group $\B_n$ and the loop braid group $\LB_n$, and both provide a rich source of representations.  In this section we explore these relationships.

\subsection{Relationship with $\B_n$} \label{ss:B2N}
As noted, the braid group $\B_n$ on $n$ strands 
%\ppm{...am I right to change $B_n$ to $\B_n$? (if so then I 
%will implement others below.) }\er{yes, I prefer this}
is isomorphic to the subgroup of
$\NB_n$ generated by $\sigma_1,\ldots, \sigma_{n-1}$.  Notice that the
absence of $\sigma_n$ and $\tau$ obviates the consideration of indices
modulo $n$ in relations (B1) and (B2).  In particular, any
representation of $\NB_n$ restricts to a representation of $\B_n$. 

In $\B_n$ define the {\em single twist}
\beq \label{eq:twist}
\gamma= \sigma_1\cdots\sigma_{n-1}. 
\eq
The center of $\B_n$ is generated by the the full twist of the $n$
strands: $\gamma^n=(\sigma_1\cdots\sigma_{n-1})^n$.  

$\B_n$ can be generated by %the single twist
$
\gamma 
$
and $\sigma_1$, as can be seen from the useful:
\begin{equation}\label{conjugator}
 \gamma^{k}\sigma_1\gamma^{-k}=\sigma_{k+1},\quad 1\leq k\leq n-2.
\end{equation}
Defining
\beq \label{eq:sigman}
\sigma_{n}^\prime:=\gamma^{n-1}\sigma_1\gamma^{1-n}=\gamma^{-1}\sigma_1\gamma=\gamma\sigma_{n-1}\gamma^{-1}
\eq
\ignore{{
\ppm{this is perhaps notationally questionable - we essentially identified 
$\B_n$ with the subgroup in $\NB_n$; but we certainly don't identify $\sigma_n$. perhaps call it something slightly different? $\sigma'_n$?} \er{Although $\sigma_n$ doesn't exist in $\B_n$ so there is no confusion I agree it is worth emphasizing.}
}}%
we find that $\sigma_1,\ldots,\sigma_{n-1},\sigma^\prime_n$ 
satisfy the (modulo $n$)
relations (B1) and (B2) above.  Setting $\tau=\gamma$ we also verify
(N1), but not (N2) in general.  However, for certain representations
we can take advantage of this close relationship to produce
representations of $\NB_n$: 

\begin{lemma} \label{lem:1}
  Let $\rho:\B_n\rightarrow GL(V)$ be any indecomposable
  finite-dimensional  representation of $\B_n$ such that $\rho(\gamma^{2n} ) = c_\rho \one_V$  for some scalar $c_\rho$ (for example any irreducible $\rho$). Then $\rho$ extends to an indecomposable representation of $\NB_n$ by
  $
  \rho(\sigma_n) = \rho(\gamma \sigma_{n-1} \gamma^{-1} )
  $
  and 
 \[
 \rho (\tau ) = %\frac{1}
                 {(c_\rho)^{-1/2n}} \rho(\gamma)  .
 \]

\end{lemma}
\begin{proof}
Since $\gamma^n$ is central, $\rho(\gamma^n ) \propto 1$ by Schur's
Lemma. Relations (B1), (B2) and (N1) are immediate.
For (N2):
$$
\rho(\tau^{2n} ) = \left(  {(c_\rho)^{-1/2n}} \rho(\gamma)   \right)^{2n}
                = \frac{1}{c_\rho} \rho(\gamma^{2n}) = Id_V  .
$$
\end{proof}
 In fact, since the braid relations are homogeneous
 we can rescale the $\rho(\sigma_i)$s by $\kappa\neq 0$ to obtain a new representation
 $\rho'(\sigma_i) = \kappa\rho(\sigma_i)$ of $\B_n$. 
 Setting $\rho'(\tau)=\rho'(\gamma)$ the rescaling
 will not affect (B1), (B2), (N1), but can be chosen to cancel the scalar in
 $\rho(\gamma^{2n})$, and then we may define $\rho(\sigma_n)$ as above.

Now let $\rho$ be any finite-dimensional representation of $\B_n$. It
is a direct sum of indecomposable representations, $\rho_i$ say, and there is a
projection to each of these in $\End_{\B_n}(V)$.
The sum of the projections each individually rescaled according to
Lemma~\ref{lem:1} gives a $\rho(\tau)$ and hence a  representation of $\NB_n$.
In particular:

\begin{theorem}\label{stdexn}
Let $\rho:\B_n\rightarrow GL(V)$ be any completely reducible 
complex representation of $\B_n$ (for example any unitary or irreducible representation).  Then there exists a $D\in\End_{\B_n}(V)$ such that defining $$\rho(\tau)=D\rho(\gamma), \quad \rho(\sigma_n)=\rho(\gamma)\rho(\sigma_{n-1})\rho(\gamma^{-1})$$ is a representation of $\NB_n$.
 
\end{theorem}
\begin{proof}
By complete reducibility, for some set of irreducibles 
$\{ W_i \}_i$ we have $V\cong\bigoplus_i W_i$.
%with $W_i$ irreducible.  
Since $\gamma^{2n}$ is central in $\B_n$, we have 
$\rho\mod_{W_i}(\gamma^{2n})=c_i\one_{W_i}$ for some $c_i\in\C$.  
Define 
$D=\bigoplus_i(c_i)^{\frac{-1}{2n}} \one_{W_i}
\in\End_{\B_n}(V)$ 
so that $(D\rho(\gamma))^{2n}=\one_V$. 
Since $D$ commutes with the operators $\rho(\sigma_i)$ for $1\leq i\leq n-1$, it also commutes with $\rho(\gamma)$. 
Hence defining $\rho(\tau)=D\rho(\gamma)$ and $\rho(\sigma_n)=\rho(\gamma^{-1}\sigma_1\gamma)$ we see by the above discussion that all $\NB_n$ relations are satisfied and hence these assignments extend $\rho$ to $\NB_n$.
\end{proof}

This approach bears some similarity with the notion of a {standard extension} found in \cite{bruetal}, so we adopt this nomenclature and refer to any representation $\rho$ of $\NB_n$ with $\rho(\tau)=A\rho(\gamma)$ and $[A,\rho(\sigma_i)]=1$ for all $i$, a \textbf{standard extension}.  Since $A=\rho(\tau\gamma^{-1})$, this operator is already in the image of $\rho$, so that $\rho(\NB_n)$ is generated by $\rho(\B_n)$ and $A$.  In particular, $\rho(\NB_n)$ is a central product of the cyclic group generated by $\rho(\B_n)$ and the central operator $A$.  From a topological perspective standard extensions do not fully exploit the $(3+1)$-dimensional nature of $\NB_n$, rather the interesting information they carry is already present in the braid group $\B_n$, which is related to $(2+1)$-dimensional topology.

\begin{remark}
By Schur's lemma a standard extension of an irreducible $\B_n$ representation has the form $\rho(\tau)=\lambda\rho(\gamma)$, where $\lambda$ is a scalar.  On the other hand, the standard extensions fit into a more general construction, observing that a representation $\varphi$ of $\B_n$ lifts trivially to $\CB_n$ by letting $\tau\in\CB_n$ act by $\varphi(\gamma)$.  Indeed, if $G=F\ltimes \langle y\rangle$ is a semi-direct product in which $y$ acts by conjugation and $y^n$ is central for some $n$ then any representation of $G$ for which the image of $y^n$ is semisimple (diagonalizable) factors over the quotient $G/\langle y^n\rangle$.  It is conceivable that more interesting representations of $\NB_n$ can be obtained from $\CB_n$ in this way, cf. \S\ref{ss:XXZ} 
\end{remark}

\begin{remark} \label{rem:indec}
Failure of complete reducibility does not preclude the existence of standard extensions: the representation of $\B_n$ defined by $\rho(\sigma_i)=J=\begin{pmatrix}1&1\\0&1
\end{pmatrix}$ for all $i$ is not completely reducible, yet $\rho(\gamma)=J^{n-1}$ commutes with $\rho(\sigma_i)$ so that $\rho(\tau)=Id=(J^{1-n})\rho(\gamma)$ is a standard extension.
\end{remark}

Of course our existence Theorem 
does not tell us how to {\em  construct} standard extensions. In this setting there is 
a computational distinction to be made between individual values of $n$ (see \S\ref{ss:low-d}), and a construction that 
starts with a braid group representation for all $n$ and 
determines a $D$ for each $n$. 
We will touch on the latter problem in \S\ref{ss:XXZ},
where we use integrable spin-chain methods.

Modular categories are a rich source of representations of $\B_n$, where motions of points in a disk lead to $\B_n$ representations on the Hilbert spaces obtained from the corresponding $(2+1)$TQFT.  Categorical constructions of $(3+1)$TQFTs suggest that there should also be a way to obtain representations of $\NB_n$ and other motion groups of 1-dimensional submanifolds of 3-manifolds, by acting on appropriate morphism spaces.

\subsection{Relationship between $\NB_n$ and $\LB_n$} $\;$ 
\label{ss:LB}

%\ppm{Should perhaps either promote the general comments on $\B_n$ and twist
%in \S\ref{ss:fromB}
%above here (since it is really relevant here too);
%or demote this section until after that one).}\er{I have implemented this}

%\medskip

The motion class $\sigma_i$ in $\NB_n$ 
can be implemented if the auxiliary circle is absent
%(that is, any motion representative of the class $\sigma_i$ has this property; 
%although isotopy equivalence is different);
as can the motion $\tau$.   This suggests a relationship between $\NB_n$ and the loop braid group $\LB_n$ associated with the motions of an array of $n$ \emph{unlinked} loops in $\R^3$ studied in \cite{BaezCransWise}.  

\medskip

It will be useful to consider for a moment a groupoid 
$\Gamma$ where $\NB_n$ and $\LB_n$ both belong (cf. e.g. \cite[Ch.2]{CrowellFox}). 
A {\em link} is an  embedding of some number of copies of the circle $S^1$ in $\R^3$
(hence a certain 1d submanifold of $\R^3$).
For example consider the link $\LL_n$ in Fig.\ref{fig1}.
Another example is the $n$-component unlink.
Let $a,b$ be two such embeddings. 
%Let $\hom(a,b)$
A motion of a link (in the motion set $\hom(a,b)$) 
is a smooth variation of one such embedding
$a$ into another $b$ over an interval of time
--- a 2d submanifold of $\R^3 \times [0,1] \subset \R^4$
`starting' at $a \subset \R^3 \times \{ 0 \}$ and 
`ending' at $b  \subset \R^3 \times \{ 1 \}$.
We may combine compatible motions in the obvious way \cite[Ch.2 \S1]{CrowellFox}.

Two motions $\gamma,\gamma'$ from $a$ to $b$ 
are {\em equivalent} if
%homeomorphic in $\R^4$.
there is a continuous family
\cite[Ch.2 \S2]{CrowellFox}
in $\hom(a,b)$ starting in $\gamma$ and
ending in $\gamma'$.
Under this equivalence the classes of $\hom(a,a)$ 
(denoted $\Gamma_a = \hom_{\sim}(a,a)$) form a group.
The classes of all motions form a groupoid $\Gamma$ with object set the set of
links.
Note that 
\begin{lemma} \label{lem:isofd}
Two groups  %$\hom_{\sim}(a,a)$
$\Gamma_a$, $\Gamma_b$ are isomorphic if the object
links are the same topological link.
\end{lemma}

The loop braid group $\LB_n$ is the motion group 
$\hom_{\sim}(a,a)$ for any $a$ that is topologically the $n$-unlink.
We may visualize $\LB_n$ as follows: arrange the $n$ loops as circles in the $xy-$plane along the $x$-axis and label them $1,\ldots, n$:  
\[
a \;\;\; = \;\;\;\;\; 
\includegraphics[width=3in]{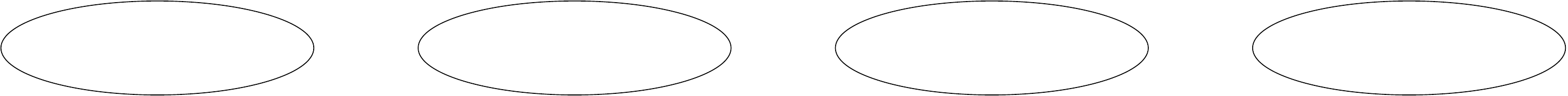}
\]
\begin{remark}
Consider for comparison
%it will be clear, cf. e.g. \cite{CrowellFox}, that 
the $n$-loop arrangement
obtained as follows. Starting from a single circle, add further circles that are rotations of it about an axis in the same plane but exterior to the circle:
\begin{equation} \label{eq:b mode}
b \;\;\; = \;\;\;\;\; 
\raisebox{-.1in}{\includegraphics[width=1.03in]{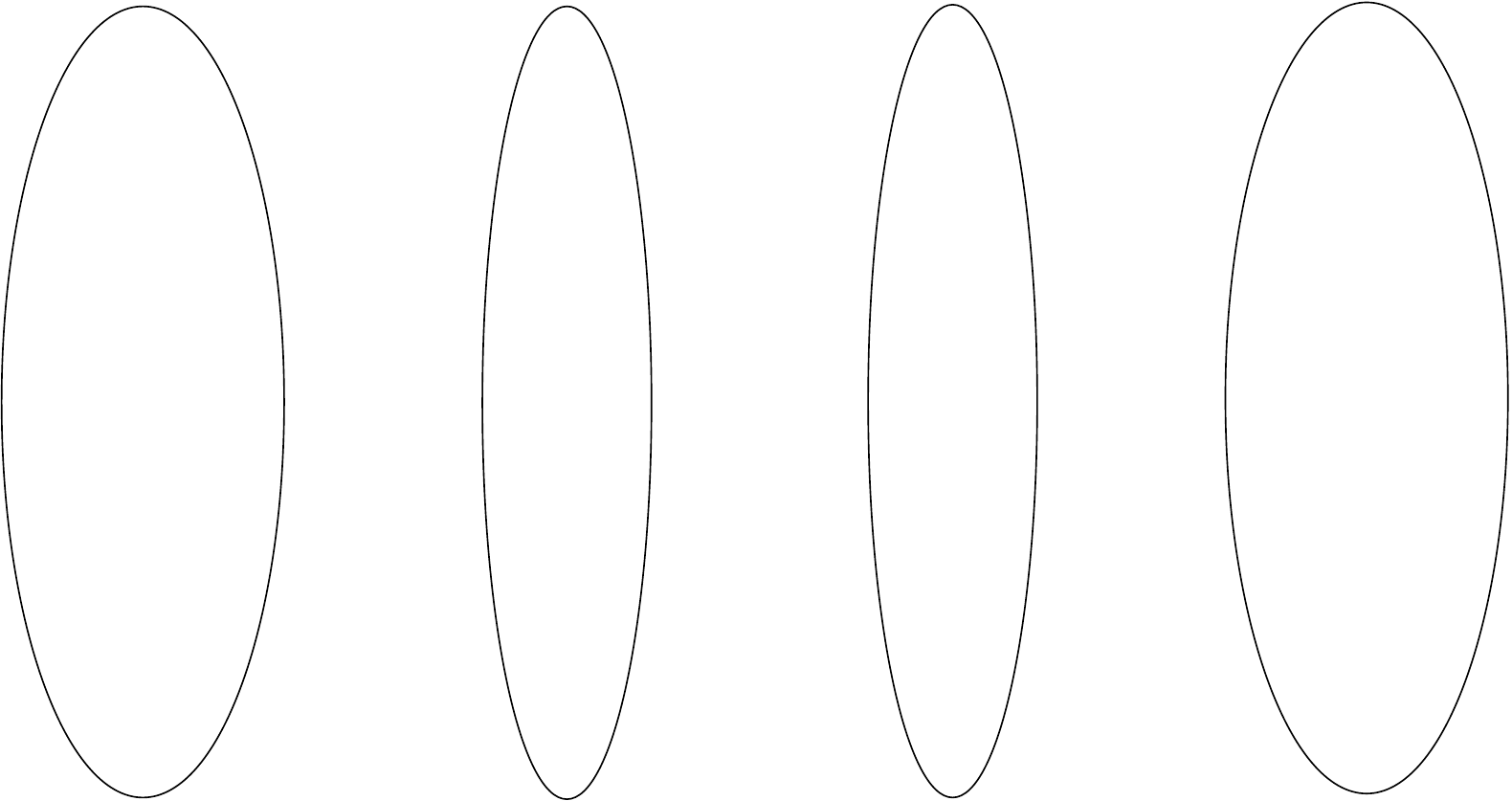}}
\end{equation}
This is then precisely as in $\LL_n$, but with the linking loop omitted.
It will be clear, cf. e.g. \cite{CrowellFox}, that
this leads to a group $\Gamma_b$ isomorphic  to $\Gamma_a$.
\end{remark}

Keeping with arrangement $a$ above, 
let $s_i$ denote the interchange of loops $i$ and $i+1$ via representative  motions like this: 
%\ppm{<- make more precise; use the fig below!?}
\[
\includegraphics[width=4.5cm]{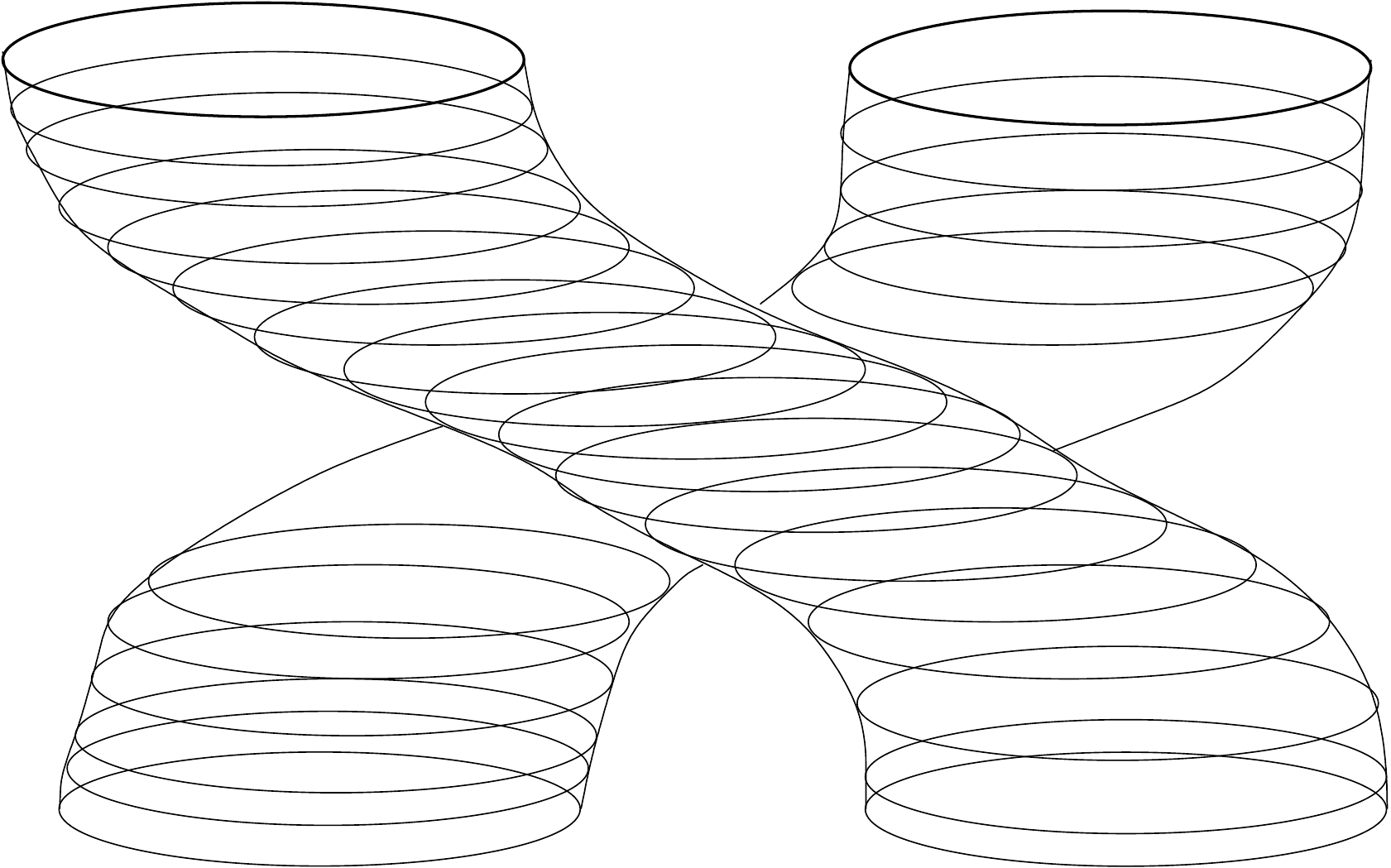}
\]
(NB. this is an overlaid `movie' view; each still is a picture of two loops in a 3d space; time goes vertically).
$\;$ 
Denote by $g_i$ the ``leapfrog" motion of passing the $i$th loop under and through the $i+1$st loop followed by sliding the $i+1$st loop into the position previously occupied by the $i$th loop, thus:

\begin{center}
    \begin{tikzpicture}[scale=0.8,every node/.style={scale=0.8}]
     \iffalse \node[] at ({-4*\rsep},{\maxy/2}) {\Large{$\sigma_{i} = $}};
      \begin{scope}[xshift = 0]
        \draw (0,0) arc (180:360:{\hrad} and {\vrad});
        \draw[dashed,color=black!80!white] (0,0) arc (180:0:{\hrad} and {\vrad});
        \draw (0,{\miny}) -- (0,{\maxy});
        \draw (1,{\miny}) -- (1,{\maxy});
        \draw (0,{\maxy}) arc(180:360:{\hrad} and {\vrad});
        \draw (0,{\maxy}) arc(180:0:{\hrad} and {\vrad});
        \node[] at ({\hrad},{\miny-2*\rsep}) {$1$};
      \end{scope}
     
      \node at (2,{\maxy/2}) {\Huge{$\cdots$}}; \fi
      \begin{scope}[xshift = 3cm]
        \draw plot [smooth,tension={\tension}] coordinates{
        ({\hrad-\hrad},0)
        ({\hrad-0.8*\hrad},{\maxy/5})
        ({2*\hrad+\rsep-0.8*\hrad},{2*\maxy/5})
        ({2*\hrad+\rsep-0.7*\hrad},{3*\maxy/5})
        ({3*\hrad+2*\rsep-1.15*\hrad},{4*\maxy/5})
        ({3*\hrad+2*\rsep-\hrad},{5*\maxy/5})};
        \draw plot [smooth,tension={\tension}] coordinates{
        ({\hrad+\hrad},0)
        ({\hrad+1.15*\hrad},{\maxy/5})
        ({2*\hrad+\rsep+0.7*\hrad},{2*\maxy/5})
        ({2*\hrad+\rsep+0.8*\hrad},{3*\maxy/5})
        ({3*\hrad+2*\rsep+0.8*\hrad},{4*\maxy/5})
        ({3*\hrad+2*\rsep+\hrad},{5*\maxy/5})};
        \draw[line width={60*0.8*\hrad},color=white] plot [smooth,tension={\tension}] coordinates{
        ({3*\hrad+2*\rsep},0)
        ({3*\hrad+2*\rsep},{\maxy/5})
        ({2*\hrad+\rsep},{2*\maxy/5})
        ({2*\hrad+\rsep},{3*\maxy/5})
        ({\hrad},{4*\maxy/5})
        ({\hrad},{5*\maxy/5})};
        \draw plot [smooth,tension={\tension}] coordinates{
        ({3*\hrad+2*\rsep-\hrad},0)
        ({3*\hrad+2*\rsep-1.1*\hrad},{\maxy/5})
        ({2*\hrad+\rsep-\hrad},{2*\maxy/5})
        ({2*\hrad+\rsep-1.1*\hrad},{3*\maxy/5})
        ({\hrad-0.9*\hrad},{4*\maxy/5})
        ({\hrad-\hrad},{5*\maxy/5})};
        \draw plot [smooth,tension={\tension}] coordinates{
        ({3*\hrad+2*\rsep+\hrad},0)
        ({3*\hrad+2*\rsep+0.9*\hrad},{\maxy/5})
        ({2*\hrad+\rsep+1.1*\hrad},{2*\maxy/5})
        ({2*\hrad+\rsep+\hrad},{3*\maxy/5})
        ({\hrad+1.1*\hrad},{4*\maxy/5})
        ({\hrad+\hrad},{5*\maxy/5})};
        \draw[dashed,color=black!80!white] plot [smooth,tension={\tension}] coordinates{
        ({\hrad-\hrad},0)
        ({\hrad-0.8*\hrad},{\maxy/5})
        ({2*\hrad+\rsep-0.8*\hrad},{2*\maxy/5})
        ({2*\hrad+\rsep-0.7*\hrad},{3*\maxy/5})
        ({3*\hrad+2*\rsep-1.15*\hrad},{4*\maxy/5})
        ({3*\hrad+2*\rsep-\hrad},{5*\maxy/5})};
        \draw[dashed, color=black!80!white] plot [smooth,tension={\tension}] coordinates{
        ({\hrad+\hrad},0)
        ({\hrad+1.15*\hrad},{\maxy/5})
        ({2*\hrad+\rsep+0.7*\hrad},{2*\maxy/5})
        ({2*\hrad+\rsep+0.8*\hrad},{3*\maxy/5})
        ({3*\hrad+2*\rsep+0.8*\hrad},{4*\maxy/5})
        ({3*\hrad+2*\rsep+\hrad},{5*\maxy/5})};
        \draw ({0.95*\hrad+\rsep},{2.5*\maxy/5}) arc (180:360:{1.05*\hrad} and {1.1*\vrad});
        \draw[dashed,color=black!20!white] ({0.95*\hrad+\rsep},{2.5*\maxy/5}) arc (180:0:{1.05*\hrad} and {1.1*\vrad});
        \draw[color=white] ({0.95*\hrad+\rsep},{2.5*\maxy/5}) arc (180:130:{1.05*\hrad} and {1.1*\vrad});
        \draw[dashed,color=black!80!white] ({0.95*\hrad+\rsep},{2.5*\maxy/5}) arc (180:130:{1.05*\hrad} and {1.1*\vrad});
        \draw[color=white] ({0.95*\hrad+\rsep+2.1*\hrad},{2.5*\maxy/5}) arc (0:50:{1.05*\hrad} and {1.1*\vrad});
        \draw[dashed,color=black!80!white] ({0.95*\hrad+\rsep+2.1*\hrad},{2.5*\maxy/5}) arc (0:50:{1.05*\hrad} and {1.1*\vrad});
        \draw[dashed,color=black!80!white] ({2*\hrad+\rsep-0.7*\hrad},{2.5*\maxy/5}) arc (180:360:{0.7*\hrad} and {0.5*\vrad});
        \draw[dashed,color=black!50!white] ({2*\hrad+\rsep-0.7*\hrad},{2.5*\maxy/5}) arc (180:0:{0.7*\hrad} and {0.5*\vrad});
        \draw (0,0) arc (180:360:{\hrad} and {\vrad});
        \draw[dashed,color=black!80!white] (0,0) arc (180:0:{\hrad} and {\vrad});
        \draw (0,{\maxy}) arc(180:360:{\hrad} and {\vrad});
        \draw (0,{\maxy}) arc(180:0:{\hrad} and {\vrad});
        \draw ({2*\hrad+2*\rsep},0) arc (180:360:{\hrad} and {\vrad});
        \draw[dashed,color=black!80!white] ({2*\hrad+2*\rsep},0) arc (180:0:{\hrad} and {\vrad});
        \draw ({2*\hrad+2*\rsep},{\maxy}) arc(180:360:{\hrad} and {\vrad});
        \draw ({2*\hrad+2*\rsep},{\maxy}) arc(180:0:{\hrad} and {\vrad});
        \node[] at ({\hrad},{\miny-2*\rsep})  {};
        \node[] at ({2*\hrad+2*\rsep+\hrad},{\miny-2*\rsep}) {} ;
      \end{scope}
      \iffalse
      \node at (6.5-\rsep,{\maxy/2}) {\Huge{$\cdots$}};
      \begin{scope}[xshift = 7cm]
        \draw (0,0) arc (180:360:{\hrad} and {\vrad});
        \draw[dashed,color=black!80!white] (0,0) arc (180:0:{\hrad} and {\vrad});
        \draw (0,{\miny}) -- (0,{\maxy});
        \draw (1,{\miny}) -- (1,{\maxy});
        \draw (0,{\maxy}) arc(180:360:{\hrad} and {\vrad});
        \draw (0,{\maxy}) arc(180:0:{\hrad} and {\vrad});
        \node[] at ({\hrad},{\miny-2*\rsep}) {$n$};
      \end{scope}\fi
    \end{tikzpicture}\end{center}
\medskip

A presentation of $\LB_n$ by abstract generators $g_i, s_i$
(using the same symbols, to indicate how the homomorphism works) is
(see e.g. \cite{BaezCransWise}):
\begin{equation} \label{eq:sss1}
g_i g_{i+1} g_i =  g_{i+1} g_i g_{i+1} ,
\qquad
g_i g_j = g_j g_i \quad |i-j|\neq 1
\end{equation}
\begin{equation}\label{eq:symmetric}
s_i^2 = 1 , 
\qquad
s_i s_{i+1} s_i = s_{i+1} s_i s_{i+1} , 
\qquad
s_i s_j = s_j s_i \qquad |i-j|\neq 1
\end{equation}
\begin{equation}\label{mixed}
s_i s_{i+1} g_i  =  g_{i+1} s_i s_{i+1}  , 
\qquad 
g_i g_{i+1} s_i  =  s_{i+1} g_i g_{i+1}  , 
\qquad 
g_i s_j  = s_j g_i \mbox{ for $ |i-j|>1$} . 
\end{equation}

\subsection{A group homomorphism}

Note that Lemma~\ref{lem:isofd} does not give an isomorphism between $\LB_n$ and $\NB_n$, and indeed they are not  isomorphic.  One way to see this is by taking their quotients by their commutator subgroups, i.e their abelianizations:  $\NB_n^{\textrm{ab}}\cong \Z\times\Z_{2n}$ whereas $\LB_n^{\textrm{ab}}\cong \Z\times \Z_2$ for $n>1$.  Note that in the degenerate case $n=1$ the group $\NB_1\cong \Z_2$ and $\LB_1=1$.
%\ppm{ %--- not obviously isomorphic. 
%(Eric can you prove they are not e.g. from dimensions of low-dimensional reps? --- meanwhile... $\NB_n$ has non-trivial centre; do we know $\LB_n$ has trivial centre? Yes, Savushkina \cite{Savushkina}.)}
However our remark on the relationship between these constructions does lead to a beautiful homomorphism.

The precise relationship is the following:

\begin{lemma} \label{lem:N2L}
\ There is a group homomorphism 
$\zeta : \NB_n \rightarrow \LB_n$ 
given by 
$$
\tau \mapsto p := s_1 s_2\cdots s_{n-1} ;
$$
$\sigma_i \mapsto g_i$ for $i<n$; and
$\sigma_n \mapsto p g_{n-1} p^{-1}$.

\end{lemma}
\begin{proof}

%\ppm{..attempt. If it works.}
%We need
It is enough to check that the relations given in Lemma \ref{relationslemma} are satisfied. Notice (B1) (for $\zeta(\sigma_1)$ and $\zeta(\sigma_2)$) as well as (B2) (for $\zeta(\sigma_1)$ and $\zeta(\sigma_j)$ with $2<j<n$) are verified by (\ref{eq:sss1}) and $\zeta(\tau)^{2n}=1$ follows from the symmetric group relations (\ref{eq:symmetric}).  We  use  (\ref{mixed}) to see that 
$$
pg_i=s_1\cdots (s_is_{i+1}g_i) s_{i+2}\cdots s_{n-1}=s_1\cdots s_{i-1}(g_{i+1}s_is_{i+1})\cdots s_{n-1}=g_{i+1}p
$$ 
proving (N1) for $i\leq n-2$.  For $i=n-1$ (N1) is true by definition: $\zeta(\sigma_n)=pg_{n-1}p^{-1}=\zeta(\tau\sigma_{n-1}\tau^{-1})$.  Moreover, $g_1=p^{-n+1}g_np^{n-1}=pg_np^{-1}$ (since $p^n=1$) which proves (N1) for $i=n$, completing the verification.

\end{proof}

\begin{theorem}
The kernel of $\zeta$ is the order $2$ central subgroup $\langle\tau^n\rangle$.
\end{theorem}
\begin{proof} This is a direct consequence of \cite[Lemma 3.1(2)]{BB16} after observing that $\zeta$ is induced by the map from the necklace $\LL_n$ to $n$ free loops that forgets the auxiliary linking circle.
\end{proof}
\iffalse
Question: This  generates the kernel? 

\begin{remark} Suppose that $\tau^k\beta$ is in the kernel of $\zeta$.  Since $\zeta(\tau^k)$ has finite order (dividing $n$), $\zeta(\beta)\in\LB_n$ also has finite order dividing $n$.  The restriction of $\zeta$ to the subgroup $\langle \sigma_1,\ldots,\sigma_{n-1}\rangle<\NB_n$ isomorphic to $\B_n$ is an isomorphism, so the generator $\sigma_n$ must appear in $\beta$. The torsion subgroup of $\LB_n$ is not known, so we cannot verify the conjecture easily.
Notice that since the subgroup of $\NB_n$ generated by the $\sigma_i$ is isomorphic to $B\tilde{A}_n$ which is torsion-free, $\beta$ is in some proper normal subgroup.

It is clear from the defining relations that the abelianization of $\LB_n$ is isomorphic to $\Z\times \Z_2$ with the image of $g_1$ generating $\Z$ and the image of $s_1$ generating $\Z_2$.  Provided $n$ is odd, $\zeta(\tau)$ is trivial.  Thus $\zeta(\beta)$, having finite order must be in the kernel of the abelianization map of $\LB_n$, i.e., in the commutator subgroup $\LB_n^\prime$.  
\end{remark}
\begin{question} \label{q:bb} %Question: 
Exactly how does this connect to the results in 
\cite[Lem.6]{BB16}?
\end{question}
%so that  we can get reps by restriction followed by pull-back. 
% \ppm{ok. but why not just ...}
 \fi
 
Note that for any representation 
 $\rho :\LB_n \rightarrow GL(V)$ 
 then  $\rho \circ 
 \zeta  %\circ \rho$ 
: \NB_n \rightarrow GL(V)$
is a representation of $\NB_n$.

\medskip 

Let us briefly review $\LB_n$
and see what this gets us in terms of representations. %\ppm{yes? just an experiment for now.} 
Present knowledge of $\LB_n$ representation theory 
is limited, but not zero (see  e.g. \cite{kmrw} for a review). 
(NB 
the motion $s_i$ is not possible in $\NB_n$).

%This gives rise to several nice points. 
Firstly there is a realisation 
of $\LB_n$ as the group of conjugating automorphisms of the 
the free group 
$F_n = \langle x_1, x_2,...,x_n \rangle$  generated by braid and permutation automorphisms
\cite{frr}:

\iffalse
\begin{equation} \label{eq:free1}
g_i : \left. \begin{array}{llll} 
x_i    & \mapsto x_{i+1} \\
x_{i+1} & \mapsto x_{i+1}^{-1} x_i x_{i+1} \\
x_j    & \mapsto x_j \;\;\; \mbox{otherwise} 
\end{array} \right. , 
\qquad\qquad
s_i : \left. \begin{array}{llll} 
x_i   & \mapsto x_{i+1} \\
x_{i+1} & \mapsto  x_i  \\
x_j    & \mapsto x_j \;\;\;\; \mbox{otherwise} 
\end{array} \right. 
\end{equation}
\fi
\begin{equation} \label{eq:free1}
g_i : \begin{cases} 
x_i    & \mapsto x_{i+1} \\
x_{i+1} & \mapsto x_{i+1}^{-1} x_i x_{i+1} \\
x_j    & \mapsto x_j \;\;\; \mbox{otherwise} 
\end{cases}
\qquad\qquad
s_i : \begin{cases} 
x_i   & \mapsto x_{i+1} \\
x_{i+1} & \mapsto  x_i  \\
x_j    & \mapsto x_j \;\;\;\; \mbox{otherwise} 
\end{cases}
\end{equation}
While this is a faithful action for $\LB_n$, the induced action of $\NB_n$ is of course not faithful.  

Local representations of $\LB_n$ are constructed in \cite{kmrw} which give rise to $\NB_n$ representations.  However, there are many local representations of $\NB_n$ that do not extend to $\LB_n$ (cf. \S\ref{ss:LR}).

\iffalse
There is also a construction based directly on moving loop observables 
around in higher gauge theory \cite{BullivantMM}.  
Many of these motions are {\em not} realisable in $\NB_n$. But some are, and we 
consider constructions based on this idea in \S\ref{s:future}.
\fi

\section{Extensions of Familiar $\B_n$ Representations}\label{s:extensions}
Having laid out the general theory in \S\ref{ss:B2N}, 
in this section we provide concrete examples of $\NB_n$ representations obtained by extending some well-known representations of $\B_n$.
\subsection{Extensions of Standard Representations}
Recall 
(see e.g. \cite{sysoeva})
that a \emph{standard representations} of $\B_n$ is $(\beta,V)$ such that $$\beta(\sigma_i)=I_{i-1}\oplus\begin{pmatrix}0&z\\1&0\end{pmatrix}\oplus I_{n-i-1}.$$ Where $I_k$ is the $k$ dimensional identity, and $z\in\C\backslash\{0,1\}$.  Notice that the image of a standard representation is a group of monomial matrices, and is therefore virtually finite.   First we will state a theorem that deals with $n\geq3$, and then we will state results for $n=2$.  We caution the reader that the established nomenclature is unfortunate as we must now discuss standard extensions of standard representations of $\B_n$.

\begin{prop}\label{std ext std rep}
For $n\geq3$, any standard extension of the standard representation from $\B_n$ to $\NB_n$ is of the form $\rho(\tau)=\lambda\beta(\gamma)$, where $\lambda\in\C$ such that $\lambda^{2n}=z^{-2(n-1)}$.
\end{prop}
\begin{proof}
From the fact that the standard representation 
is irreducible for $n\geq 3$ \cite[Lemmas 5.3 and 5.4]{sysoeva}, we have that $\rho(\tau)=\lambda\beta(\gamma)$ for $\lambda\in\C\backslash\{0\}$. The fact that $\rho(\tau)^{2n}=\lambda^{2n}\beta(\gamma)^{2n}=I_{n}$, and $\beta(\gamma)^{2n}=z^{2(n-1)}I_n$ gives us that $\lambda^{2n}=\displaystyle z^{-2(n-1)}$.
\end{proof}

Now considering $n=2$, we want $A=\begin{pmatrix}a&b\\c&d\end{pmatrix}$ such that $AZ=ZA$ (where $\beta(\sigma_1)=Z=\begin{pmatrix}0&z\\1&0\end{pmatrix}$), $(AZ)^{4}=I_2$, and $(AZ)^2Z(AZ)^{-2}=Z$. The last equation comes from the fact that we want $\tau\sigma_2\tau^{-1}=\sigma_1$, and we are defining the image of $\sigma_2$ to be the image of $\tau\sigma_1\tau^{-1}$. From $AZ=ZA$ we get that $a=d$ and $b=zc$. Using this along with the other two equations, we get the following possibilities for $\rho(\tau)=AZ$:
$$\left\{\pm I_2,\pm\begin{pmatrix}1&0\\0&{i}(z)^{-1}\end{pmatrix},\xi_4\begin{pmatrix}0&\sqrt{z}\\(\sqrt z)^{-1}&0\end{pmatrix},\pm\frac{1}{2}\begin{pmatrix}1\pm i&(1\mp i)\sqrt z\\ (1\mp i)\sqrt z\iv&1\pm i\end{pmatrix}\right\},$$ where $\xi_4$ is a choice of 4th root of unity.

Next let us consider non-standard extensions of the standard representation. This means we want a representation $(\phi,V)$, and $T\in \End(V)$ such that $\phi(\tau)=T$ and $T\neq \lambda\beta(\gamma)$ (where $\lambda$ is as defined in proposition \ref{std ext std rep}). We know that we need $T^{2n}=I_V$, $T\beta(\sigma_i)=\beta(\sigma_{i+1})T$ for all $i=1,\dots,n-2$, and that $T^2\beta(\sigma_{n-1})=\beta(\sigma_1)T^2$. Let $T=(t_{i,j})_{i,j=1}^n$. The later two relations give us that $t_{i,j}=0$ if $j\not\equiv i-1\mod n$ and $t_{i,i-1}=t_{2,1}$ for all $i=2,\dots,n-1$. Hence $T$ has the following block form: $\begin{pmatrix}0& a\\ tI_{n-1} & 0\end{pmatrix}.$ Then $T^{2n}=I_V$ gives us that $a=t^{-(n-1)}$, and therefore $T=\begin{pmatrix}0&t^{-n+1}\\tI_{n-1}&0\end{pmatrix}$. If $t^{2n}=z^{-2(n-1)}$, we would have a standard extension. Hence, $t^{2n}\neq z^{-2(n-1)}$ would give us a non-standard extension. This gives us the following:

\begin{theorem}
For $n\geq3$, a representation, $\phi$, of $\NB_n$ is an extension of the standard representation, $\beta$, of $\B_n$ if $\phi(\sigma_i)=\beta(\sigma_i)$ for $i=1,\dots,n-1$, $\phi(\tau)=\begin{pmatrix}0&t^{-n+1}\\tI_{n-1}&0\end{pmatrix}$ (for $t\neq0$), and $\phi(\sigma_n)=\phi(\tau\sigma_{n-1}\tau\iv)$.\\ It should be noted that if $t^{2n}\neq z^{-2(n-1)}$, then the representation $\phi$ is not a standard extension of $\beta$, i.e.\ the image of $\tau$ is not a rescaling of that of the single twist $\gamma$.
\end{theorem}

%\section{Extensions of the Burau Representation}
\subsection{Extensions of the Reduced Burau representation}
As a reminder, the reduced Burau representation $\rho$ of $\B_n$ is a $n-1$ dimensional representation defined as follows:
$$\rho(\sigma_1)=\begin{pmatrix}-t&0& \\ -1&1& \\ & & I_{n-3}\end{pmatrix},\rho(\sigma_i)=\begin{pmatrix}I_{i-2} & & & \\ & 1 & -t & 0 & \\ & 0 & -t & 0 & \\ & 0 & -1 & 1 & \\ & & & & I_{n-i-2} \end{pmatrix},\rho(\sigma_{n-1})=\begin{pmatrix}I_{n-3} & & \\ & 1 & -t \\ & 0 & -t\end{pmatrix}$$ where $t$ is a nonzero complex number. The reduced Burau representation is irreducible if $1+t+t^2+\cdots+t^{n-1}\neq0$. This means (as stated in remark 2.3)  that any standard extension has $\rho(\tau)=\lambda\rho(\gamma)$.

\begin{prop}
Any standard extension $\rho$ of $\NB_n$ of the reduced Burau representation (with $1+t+\cdots+t^{n-1}\neq0$) has the form $\rho(\tau)=\lambda\rho(\gamma)$ where $\lambda^{2n}=t^{-2n}$
\end{prop}
\begin{proof}
From the above, we have that $\rho(\tau)=\lambda\rho(\gamma)$ for some scalar $\lambda$. From the fact that $\rho(\tau)^{2n}=\lambda^{2n}\rho(\gamma)^{2n}=I_V$, and $\rho(\gamma)^{2n}=t^{2n}I_V$. We are given that $\lambda^{2n}=t^{-2n}$
\end{proof}

%%%%%%%%%%%%%%%%%%%%%%%%%%%%%%%%%%%%%%%%%%%%%%%%%%%%%%%%%%%%%%%%%%%%%%%%%%%%%
%%%%%%%%%%%%%%%%%%%%%%%%%%%%%%%%%%%%%%%%%%%%%%%%%%%%%%%%%%%%%%%%%%%%%%%%%%%%%dd

\subsection{Extensions of the Lawrence-Krammer-Bigelow Representation} Let $V$ be a $\binom{n}{2}$ dimensional vector space with basis $v_{i,j}$ ($1\leq i,j\leq n$). Assuming that the order of the indices do not matter, and that $t,q$ are nonzero complex numbers, the Lawrence-Krammer-Bigelow (LKB) representation is defined as:
\begin{align*}
\sigma_iv_{i,i+1} &= tq^2v_{i,i+1} & \\
\sigma_iv_{j,k} &= v_{j,k} & \text{for }\{i,i+1\}\cap\{j,k\}=\emptyset\\
\sigma_iv_{i+1,j}&=v_{i,j} & \text{ for } j\neq i,i+1\\
\sigma_iv_{i,j}&=tq(q-1)v_{i,i+1}+(1-q)v_{i,j}+qv_{i+1,j} & \text{ if }i+1<j\\
\sigma_iv_{j,i}&=(1-q)v_{j,i}+qv_{j,i+1}+q(q-1)v_{i,i+1} & \text{ if }j<i
\end{align*}
It can be computed that $\gamma v_{i,j}=\left\{\begin{array}{cc}
tq^2{v}_{i,i+1} & \text{if }j=n\\ q^2v_{i+1,j+1} & \text{if }j<n
\end{array}\right..$ Repeating this gives that $\gamma^nv_{i,j}=tq^{2n}v_{i,j}$. Therefore, a standard extension to $\NB_n$ given by $\tau\mapsto\kappa\gamma$ where $\kappa$ is a scalar, we find $\kappa=\omega_{2n}(t^{-1/n}q^{-2})$, with $\omega_{2n}$ a $2n$-th root of unity. For $n=3$ and $n=4$, these are the only standard extensions. Notice that if $\omega_{2n}$ is an $n$th root of unity, $\tau^n$ would be in the kernel. Thus the standard extension would not be faithful.
\subsubsection{Nonstandard Extensions}
For $n=2$, since any LKB representation is 1 dimensional, any extension is standard.\\ For $n=3$, with the additional assumption that $q\neq1$ and $\alpha$ a choice of cube root of $\pm t^{-1}$, the following give a nonstandard extension of the LKB representation: $$\tau\mapsto\alpha\begin{pmatrix}0&(q^2-q+1)q^{-2}&-(q-1)q^{-2}\\0&-(q-1)q^{-1} &  q^{-1}\\  tq^2&(q-1)(tq^2-q+1)q^{-1} & (q-1)q^{-1}\end{pmatrix}.$$ If we choose $\alpha$ to be the cube root of $t^{-1}$, then the extension is not faithful (as $\tau^3$ would be in the kernel).\\ For $n=4$, again, let $q\neq1$ and $\beta\in\{\pm\sqrt{\pm t},(-t^2)^{\frac{1}{4}}\}$. Then we obtain a nonstandard extension of the LKB representation by having the image of $\tau$ be
{\small$$\beta\begin{pmatrix}
0&0&(q^4t)^{-1}(q^3-q+1)&0&-(q^4t)^{-1}p & -(q^4t)^{-1}p\\
0&0&-(q^3t)^{-1}p & 0 & (q^3t)^{-1}(q^2-q+1) & -(q^3t)^{-1}p\\
0&0&-(q^2t)^{-1}p & 0 & -(q^2t)^{-1}p & (q^2t)^{-1}\\
q^2 & 0 & (q^3t)^{-1}p(q^3t-q+1) & 0 & (q^3t)^{-1}(q^3-2q^2+2q-1) & -(q^3t)^{-1}p^2\\
0 & q^2 & (q^2t)^{-1}p(q^3t-q+1) & 0 & -(q^2t)^{-1}p^2 & (q^2t)^{-1}p\\
0&0& -(qt)^{-1}p^2 & q^2 & (qt)^{-1}(q^3t-q^2(t+1)+2q-1) & (qt)^{-1}p
\end{pmatrix} ,$$} where $p=q-1$. Similar to the case with $n=3,$ if we choose $\beta=\pm\sqrt{\pm t}$, then $\tau^4$ is in the kernel of the extension. Hence it would not be faithful.

%%%%%%%%%%%%%%%%%%%%%%%%%%%%%%%%%%%%%%%%%%%%%%%%%%%%%%%%%%%%%%%%%%%%%%%%%%%%%
%%%%%%%%%%%%%%%%%%%%%%%%%%%%%%%%%%%%%%%%%%%%%%%%%%%%%%%%%%%%%%%%%%%%%%%%%%%%%
\section{Local Representations} \label{ss:LR}

One source of matrix representations of $\B_n$ is through braided vector spaces (BVS): these are pairs $(R,V)$ where $V$ is a vector spaces and $R\in\Aut(V^{\ot 2})$ satisfies the Yang-Baxter equation (on $V^{\ot 3}$)
$$(R\ot I_V)(I_V\ot R)(R\ot I_V)=(I_V\ot R)(R\ot I_V)(I_V\ot R).$$
The assignment $\rho^R(\sigma_i)=I_V^{\ot (i-1)}\ot R\ot I_V^{\ot (n-i-1)}$ then gives a representation of $\B_n$ on $V^{\ot n}$.  This is an example of a \emph{local} representation: each generator has non-trivial action only on two (adjacent) copies of $V$.  
\begin{remark}Note that in general $\rho^R$ may not lift to a representation of $\LB_n$: in \cite[Proposition 3.3]{kmrw} BVSs of \emph{group-type} are shown to lift to a loop braided vector space but the general case is open.
\end{remark}

By Theorem \ref{stdexn}, $\rho^R$ has a standard extension (as long as it is completely reducible).  As above, the image of the standard extension of $\rho^R$ to $\NB_n$ does not carry much more information than $\rho^R$ itself.  However, there is another extension of $\rho^R$, using the BVS obtained from the symmetric group.  Namely, define the flip operator $P(x\ot y)=y\ot x$ on $V\ot V$.  Then we have:

\begin{theorem}\label{localextension} Suppose $(R,V)$ is a BVS and $\rho^R$ the corresponding $\B_n$ representation.
 For $n\geq 3$, setting $\rho^R(\tau)=(P\ot I_V^{\ot n-2})\cdots(I_V^{\ot n-2}\ot P)$ defines an extension of $\rho^R$ to $\NB_n$. 
\end{theorem}
\begin{proof}
% It is enough to check\footnote{AK: I have proven this, although I'm not sure if $n=2$ works in general} the following relations for $n=4$:
% \begin{enumerate}
%  \item $\tau\sigma_1\tau^{-1}=\sigma_2$
%  \item (B1) and (B2) for $i=n,n-1$
% \end{enumerate}

Again using Lemma \ref{relationslemma} it is enough to check (N1) and (N2): the fact that $(R,V)$ is a BVS gives (B1) for $i=1$ and $(B2)$ for $i=1$ and $3\leq j\leq n-1$ immediately.

 From our definition of $\rho^R(\tau)$, we have that $(N2)$ is satisfied (in fact, $\rho^R(\tau)$ has order $n$). 
 
 The key computation is to show that (N1) holds.  For this it is sufficient to show that 
 $$(P\ot I)(I\ot P)(R\ot I)(P\ot I)(I\ot P)=(I\ot R).$$
 Comparing the two operators on a pure tensor of basis elements $v_1\ot v_2\ot v_3$ we obtain $$(P\ot I)(I\ot P)[R(v_2\ot v_3)\ot v_1]\quad \text{and} \quad v_1\ot R(v_2\ot v_3)$$ for the left- and right-hand sides respectively, which are clearly equal.  This completes the proof.

\end{proof}
\begin{remarks}\begin{itemize}
    \item Obviously the operator $\rho^R(\tau)$ is not local in the strict sense: it acts non-trivially on all tensor factors.  However, its action does not mix vectors within the tensor factors, it only permutes them globally. 
    \item  We think the following is an interesting question: given $R$, how much bigger is the image $\rho^R(\NB_n)$ than $\rho^R(\B_n)$?  Note that while the subgroup generated by $\sigma_1,\ldots,\sigma_n$ has index $2n$ in $\NB_n$ \cite{BB16} $\B_n=\langle\sigma_1,\ldots,\sigma_{n-1}\rangle$ has infinite index.  If $|\rho^R(\B_n)|<\infty$ is  $|\rho^R(\NB_n)|<\infty$?
    \item In the degenerate $n=2$ case we have $\B_2\cong\Z$, so that \emph{any} $R \in\Aut(V^{\ot 2})$ gives a representation of $\B_2$.  
Setting $\rho^R(\tau)=P$ defines an extension to $\NB_2$ if $R$ is symmetric in the standard product basis of $V\ot V$: we have that $\rho^R(\sigma_1)=R$, $\rho^R(\tau)=P$, and $\rho^R(\sigma_2)=P R P$. The only relation that needs checking is $\rho^R(\sigma_1\sigma_2\sigma_1)=\rho^R(\sigma_2\sigma_1\sigma_2)$ i.e.
$RPRPR=PRPRPRP$ which is satisfied if $R$ is symmetric i.e.$PRP=R$. Note that $(B1)$ does not hold for every $R$ satisfying the Yang-Baxter equation, as the following example illustrates.
\end{itemize}
\end{remarks}
 \begin{example}\label{isingexample}
 Consider $\dim V=2$ and $R=\displaystyle\frac{1}{\sqrt2}\begin{pmatrix}1&0&0&1\\0&1&-1&0\\0&1&1&0\\-1&0&0&1\end{pmatrix}$. This gives us that $$\rho^R(\sigma_1\sigma_2\sigma_1)(e_2\ot e_1)\neq\rho^R(\sigma_2\sigma_1\sigma_2)(e_2\ot e_1).$$ Thus mapping $\rho^R(\tau)=P$ does not define an extension of $\rho^R$ from $\B_2$ to $\NB_2$.
\end{example}

For the matrix $R$ of this example we have verified the following conjecture for $n\leq 5$, showing that the image of $\NB_n$ can be significantly larger than that of $\B_n$.
\begin{conj}
Given the above $R$ and $n\geq3$, $|\rho^R(\NB_n)|=n|\rho^R(B\tilde{A}_n)|=n2^n|\rho^R(\B_n)|$.
\end{conj}

%%%%%%%%%%%%%%%%%%%%%%%%%%%%%%%%%%%%%%%%%%%%%%%%%%%%%%%%%%%%%%%%%%%%%%%%%%%%%
\subsection{Gaussian Braided Vector Spaces}  From the above discussion, any unitary BVS provides a local representation of $\B_n$, which can be extended to a representation of $\NB_n$ in two ways: (1) the standard extension (Lemma \ref{lem:1}) and (2) the $n$-cycle extension of Theorem \ref{localextension}.  In this subsection we consider extensions of the Gaussian representations first studied by Jones \cite{jones89} and analyzed in \cite{Gaussian}.  As the matrix representations are somewhat unwieldy, we take a more algebraic approach.

As in \cite{Gaussian}, we define $ES(m,n-1)$ as the algebra generated by $u_1,\dots, u_{n-1}$ with the relations $u_i^m=1$, $u_{i}u_{i+1}=q^2u_{i+1}u_i$, and $u_iu_j=u_ju_i$ if $|i-j|>1$ where $q=\begin{cases} e^{2\pi i/m}, & \text{ if }m\text{ odd}\\ e^{\pi i/m}, & \text{ if }m\text{ even}\end{cases}.$ Setting $\varphi_n(\sigma_i)=\sum_{j=0}^{m-1}q^{j^2}u_i^j$ defines a homomorphism $\varphi_n:\B_n\to ES(m,n-1)$. To get a braided vector space from $ES(m,n-1)$ it is enough to find a vector space $V$ and $U\in \Aut(V^{\otimes2})$ such that the map $u_i\mapsto I_V^{\otimes i-1}\otimes U\otimes I_V^{\otimes n-i-1}$ defines a representation of $ES(m,n-1)$ on $V^{\otimes n}$. Let $V\cong\C^m,$ with standard basis $\{e_i|0\leq i\leq m-1\}$. Define $e_{i+m}=e_i$, and $U\in\End(V^{\otimes 2})$ by $U(e_i\otimes e_j)=q^{j-i}e_{i+1}\otimes e_{j+1}$. In \cite{Gaussian} it was shown that $u_i\mapsto U_i:=I^{\otimes i-1}\otimes U\otimes I^{\otimes n-i-1}$ gives a $*$-algebra homomorphism (where $u_i^*=u_i^{-1}$) from $ES(m,n-1)$ to $\End(V^{\otimes n})$. It was also shown that $R:=\frac{1}{\sqrt m}\sum_{j=0}^{m-1}q^{j^2}U^j$ is a unitary operator. Composing with $\varphi_n$ we obtain a unitary representation $\Phi_n:\B_n\to\Aut(V^{\otimes n})$, where $\Phi_n(\sigma_i)=\frac{1}{\sqrt m}\sum_{j=0}^{m-1}q^{j^2}U_i^j$.

To extend this idea to $\NB_n$ first we extend $ES(m,n-1)$ to another algebra, ${NES}(m,n)$. We define ${NES}(m,n)$ to be the algebra generated by $u_1,\dots,u_{n-1},t$ subject to the following relations:
\begin{enumerate}
\item $u_i^m=1=t^{n}$
\item $u_iu_{i+1}=q^2u_{i+1}u_i$ for all $1\leq i\leq n-2$
\item $u_iu_j=u_ju_i$ if $|i-j|\not= 1$,
\item $tu_it^{-1}=u_{i+1}$ for all $1\leq i\leq n-2$
\end{enumerate} where $q$ is either an $m$th or $2m$ root of unity as above. 

Notice that $NES(m,n)$ is \emph{nearly} a semidirect product of $ES(m,n-1)$ with $\Z_n$, except that $tu_{n-1}t^{-1}$ is not in $NES(m,n-1)$.  To make the connection to $\NB_n$ clearer, and to remedy this defect we introduce a useful auxiliary generator to obtain a presentation with more familiar modulo $n$ relations: 
\begin{lemma}\label{modnlemma}
If we define $u_n:=tu_{n-1}t^{-1}$ then $u_n$ satisfies (1) above, relations (2) and (4) above hold modulo $n$ and the condition in (3) can be replaced with $|i-j|\not\equiv 1\mod n$.
\end{lemma}
\begin{proof}
Since $t^{n-2}u_1t^{-n+2}=u_{n-1}$ and $t^n=1$ we have $t^{n-1}u_1t^{1-n}=u_{n}$ proving that (4) holds modulo $n$.  Next we see that $$u_{n-1}u_n=tu_{n-2}u_{n-1}t^{-1}=q^2tu_{n-1}t^{-1}tu_{n-2}t^{-1}=q^2u_{n-1}u_n$$ as desired, with $u_nu_{1}=q^2u_1u_n$ verified similarly.  For (3) it is enough to check that $u_n$ commutes with $u_{n-2}$ (with $n\geq 4$).  This is also straightforward: 
$$u_{n}u_{n-2}=tu_{n-1}t^{-1}tu_{n-3}t^{-1}=tu_{n-3}u_{n-1}t^{-1}=u_{n-2}u_n.$$
\end{proof}
Observe that the algebra $NES(m,n)$ is a finite dimensional semisimple algebra over $\Q(q)$.  Indeed, $NES(m,n)$ is essentially a group algebra.  Next we show that $\NB_n$ admits a representation in $NES(m,n)$.
\begin{theorem}
The map $\hat\varphi_n:\NB_n\to NES(m,n)^*$ by $\sigma_i\mapsto R_i(m)=\frac{1}{\sqrt m}\sum_{j=0}^{m-1}q^{j^2}u_i^j$ and $\tau\mapsto t$ is a group homomorphism.
\end{theorem}
\begin{proof}
As shown in \cite[Proposition 3.1]{Gaussian}, the relation $\hat\varphi_n(\sigma_1\sigma_2\sigma_1)=\hat\varphi_n(\sigma_2\sigma_1\sigma_2)$, and $\hat\varphi_n(\sigma_1\sigma_j)=\hat\varphi_n(\sigma_j\sigma_1)$ for $2<j<n$ are true. From the definition of $t$, $\hat\varphi_n(\tau\sigma_i\tau^{-1})=\hat\varphi_n(\sigma_{i+1})$ and $\hat\varphi_n(\tau)^{2n}=1.$ Thus we have $\hat\varphi_n$ is a representation of $\NB_n$ into $NES(m,n)$.
\end{proof}

To obtain a representation of $\NB_n$ again let $V\cong\C^m$, with standard basis $\{e_0,\dots,e_{m-1}\}$ with $e_{i+m}=e_i$. Now define $U,T\in\text{End}(V^{\otimes2})$ by $U(e_i\otimes e_j)=q^{j-i}e_{i+1}\otimes e_{j+1}$ and $T(e_i\otimes e_j)=e_j\otimes e_i$.  Further define, for any $n\geq 2$, elements of $\Aut(V^{\otimes n})$:
$$X:=(T\otimes I^{\otimes n-2})(I\otimes T\otimes I^{n-3})\cdots(I^{\otimes n-2}\otimes T)$$ and $U_i:=I^{\otimes i-1}\otimes U\otimes I^{\otimes n-i-1}$ for each $1\leq i\leq n-1$.  Notice that $$X(e_{i_1}\otimes\cdots\otimes e_{i_n})=e_{i_n}\otimes e_{i_1}\otimes\cdots\otimes e_{i_{n-1}}.$$
\begin{prop}
The map $\Psi$ given on generators by $u_i\mapsto U_i$ and $t\mapsto X$ defines a representation of $NES(m,n)$ on $V^{\otimes n}$.  

\iffalse and $U_n(e_{i_1}\otimes\cdots\otimes e_{i_n})=q^{i_1-i_n}e_{i_1+1}\otimes e_{i_2}\otimes e_{i_{n-1}}\otimes e_{i_n+1}$.
\fi

\end{prop}
Observe that we only need to verify relations (1)-(4) above.
\begin{proof}  
It is clear that $U_i$ commutes with $U_j$ if $|i-j|\neq 1$ proving (3). Also, $$U^m(e_i\otimes e_j)=\begin{cases} q^{(m-j)(j-i)}q^{(j-i-m)(j-i)}q^{(j-i)i}e_i\otimes e_j=e_i\otimes e_j &i\leq j\\
q^{(m-i)(j-i)}q^{(i-j-m)(j-i)}q^{(j-i)j}e_i\otimes e_j=e_i\otimes e_j&i>j.\end{cases}$$
Notice that $X$ has order $n$, hence (1) is satisfied. A straightforward calculation shows that $XU_iX\iv=U_{i+1}$, i.e. (4). From this we see that for (2) it is sufficient to check the last relation on $NES(m,n)$ for $U_1,U_2$ with $n=3$, which has been verified in \cite{RW}: $U_1U_2(e_i\otimes e_j\otimes e_k)
=q^2U_2U_1(e_i\otimes e_j\otimes e_k).$
\end{proof}

Observe that $\Psi\circ\hat{\varphi}:\NB_n\rightarrow \Aut(V^{\ot n})$ gives a local representation of $\NB_n$.  The case $m=2$ has the realisation given in Example \ref{isingexample}, and conjecturally has finite image.
More generally, in \cite{Gaussian} it is shown that the restricted image $\hat{\varphi}(\B_n)$ in $ES(m,n-1)$ is finite. We wish to follow a similar approach to show that the image of $\hat\varphi_n(\NB_n)$ is also finite. 
Notice that the monomials in $NES(m,n)$ have the following normal form: $t^\alpha u_1^{\alpha_1}\cdots u_n^{\alpha_n}$ where $0\leq \alpha<n$ and  $0\leq\alpha_i<m$. In fact, we see that these $n(m)^n$ monomials form a basis for $NES(m,n)$ over $\Q(q)$.   
The structure of $NES(m,n)$ is more complicated than $ES(m,n-1)$, which is actually simple for $n$ odd and has exactly $m$ simple components for $n$ even \cite{RW}.
We  let $\hat\varphi_n(\NB_n)\subset NES(m,n)$ act on the span of $\hat U=\{u_1^{\alpha_1}\cdots u_n^{\alpha_n}\}$ by conjugation.  Since conjugation by $t$ obviously permutes this spanning set, 
\iffalse \amk{I may be over thinking this, $t$ conjugates $u_1^{\alpha_1}\cdots u_n^{\alpha_n}$ to $q^{2\alpha_n(\alpha_{n-1}-\alpha_1)}u_1^{\alpha_n}u_2^{\alpha_1}\cdots u_n^{\alpha_{n-1}}$. So the scalar is making me wonder if we should change $\hat U$ to be $\{q^ku_1^{\alpha_1}\cdots u_n^{\alpha_n}\}$. \er{It is fine to do so--but we still have that $t$ acts monomially (permutation up to constants).  In any case, $\CB_n$ has finite index in $\NB_n$ so the same is true of its image.  So we can restrict.} Then $t$ would permute the spanning set.}  \fi
we first show that the conjugation action of $R_i(m)$ also permutes this set.  The same approach as in \cite{Gaussian} works here: (note we may omit the scalar $\frac{1}{\sqrt m}$ in $R_i(m)$ in these calculations):

\begin{eqnarray*}
qu_i\iv u_{i+1}R_i(m) & = & qu_i\iv u_{i+1}\sum_{j=0}^{m-1}q^{j^2}u_i^j  =  q\iv u_{i+1}u_i\iv\sum_{j=0}^{m-1}q^{j^2}u_i^j\\
 & = & q\iv\sum_{j=0}^{m-1}q^{j^2}u_{i+1}u_i^{j-1}  =  q\iv\sum_{j=0}^{m-1} q^{j^2}(q^{-2(j-1)})u_i^{j-1}u_{i+1}\\
 & = & \sum_{j=0}^{m-1}q\iv q^{j^2}q^{-2j+2}u^{j-1}_iu_{i+1}\\
 & = & \left(\sum_{j=0}^{m-1}q^{(j-1)^2}u_i^{j-1}\right)u_{i+1}=R_i(m)u_{i+1}\end{eqnarray*} and 
 \begin{eqnarray*}
 q\iv u_{i-1}u_iR_i(m) & = & qu_iu_{i-1}\sum_{j=0}^{m-1}q^{j^2}u_i^j\\
 & = & qu_i\sum_{j=0}^{m-1}q^{j^2}u_{i-1}u_i^j=qu_i\sum_{j=0}^{m-1}q^{j^2}q^{2j}u_i^ju_{i-1}\\
 & = & \sum_{j=0}^{m-1}qq^{j^2}q^{2j}u_i^{j+1}u_{i-1}\\
 & = & \left(\sum_{j=0}^{m-1}q^{(j+1)^2}u_i^{j+1}\right)u_{i-1}=R_i(m)u_{i-1}.
\end{eqnarray*}
This shows that $R_i(m)u_{i+1}R_i(m)\iv=qu_i\iv u_{i+1}$ and $R_i(m)u_{i-1}R_i(m)\iv=q\iv u_{i-1}u_i$. Thus conjugation by $R_i(m)$ permutes the spanning set $\hat U$ up to scalars that are roots of unity (i.e. powers of $q$).  Thus $\hat\varphi_n(\NB_n)$ is finite modulo the center.
The subalgebra of $NES(m,n)$ generated by $\hat\varphi_n(\NB_n)$ is semisimple, so that the faithful representation of $\hat\varphi_n(\NB_n)$ on $NES(m,n)$ decomposes into full matrix algebras.  Thus any element $x$ of the center of $\hat\varphi_n(\NB_n)$ acts via a scalar matrix on each irreducible subrepresentation.  But since the generators of $\hat\varphi_n(\NB_n)$ have determinant a root of unity (of degree $m$ or $n$), the scalar $x$ is also a root of unity of degree only depending on $m$ and $n$ (indeed the degree of each irreducible representation depends only on $m,n$).  Thus the center of $\hat\varphi_n(\NB_n)$ is a finite group and has finite index, so $\hat\varphi_n(\NB_n)$ is a finite group.

%%%%%%%%%%%%%%%%%%%%%%%%%%%%%%%%%%%%%%%%%%%%%%%%%%%%%%%%%%%%%%%%%%%%%%%%%%%%%
\subsection{Quaternionic Representation} Similar to the Gaussian Braided Vector Space, the idea is to take a finite group (in this case $Q_8$) and consider the group algebra with $n$ copies of the group, where the generators will interact with `close' neighbours in a specific way, but commute with`far' neighbours. Let $q=e^{2i\pi/6}$ and $[,]$ denote the group commutator. Similar to the Gaussian case, we take the algebra $Q_n$ defined in \cite{quat}, and define $\mathcal Q_n$ to be (almost) a semi-direct product of $Q_n$ with $\Z_n$ the algebra generated by $t,u_1,\dots,u_{n-1},v_1,\dots,v_{n-1}$ with the following relations:
\begin{enumerate}
\item $u_i^2=v_i^2=-1$ for all $i$,
\item $[u_i,v_j]=-1$ if $|i-j|<2$,
\item $[u_i,v_j]=1$ if $|i-j|\geq2$,
\item $[u_i,u_j]=[v_i,v_j]=1=t^{n}$,
\item $tu_it\iv=u_{i+1}$, and $tv_it\iv=v_{i+1}$ for all $i$.
\end{enumerate} As in the Gaussian case, we have the following lemma.
\begin{lemma} Defining, in $\mathcal Q_n$, $u_n:=tu_{n-1}t^{-1}$ and $v_n:=tv_{n-1}t^{-1}$, $v_n,u_n$ satisfy relations $(1)-(5)$ with indices $\mod n$ (defining $v_0=v_n$ and $v_{n+1}=v_1$).
\end{lemma} 
\begin{proof} We must check that $u_n$ and $v_n$ also satisfy the relations $(1)-(5)$. For $(1)$, note that $u_n^2=(tu_{n-1}t^{-1})^2=tu_{n-1}^2t^{-1}=-1=tv_{n-1}^2t^{-1}=(tv_{n-1}t^{-1})^2=v_n^2.$ The relation $(5)$ follows from the definition of $u_n$,$v_n$ and $t^n=1$; $tu_nt^{-1}=t(tu_{n-1}t^{-1})t^{-1}=t(t^{n-1}u_1t^{-n+1})t^{-1}=t^nu_1t^{-n}=u_1$ and similarly $tv_nt^{-1}=v_1.$ For $(2)$ and $(3)$, we will first consider $[u_n,v_j]$. Doing so gives the following:
\begin{eqnarray*}
[u_n,v_j] = [tu_{n-1}t^{-1},v_j] & = & tu_{n-1}t^{-1}v_jtu_{n-1}^{-1}t^{-1}v_j^{-1}\\
 & = & tu_{n-1}v_{j+1}u_{n-1}^{-1}v_{j+1}^{-1}t^{-1}\\
 & = & t[u_{n-1},v_{j+1}]t^{-1}.
\end{eqnarray*}
Similarly $[u_j,v_n]=t[u_{j+1},v_{n-1}]t^{-1}$. Thus the relation $(3)$ holds for $|i-j|\mod n\geq2$. For $(2)$ we need to check the above equation, $j=1,n-1,n$.
\begin{align*} [u_n,v_1] & = [t^{-1}u_1t,v_1]=t^{-1}[u_1,v_2]t=-1\\
[u_n,v_{n-1}] & =  [tu_{n-1}t^{-1},tv_{n-2}t^{-1}]=t[u_{n-1},v_{n-2}]t^{-1}=-1\\
[u_n,v_n] & =  [tu_{n-1}t^{-1},tv_{n-1}t^{-1}]=t[u_{n-1},v_{n-1}]t^{-1}=-1.
\end{align*}
Hence $(2)$ holds for $|i-j|\mod n<2$. Lastly $(4)$ holds from $[u_n,u_j]=t[u_{n-1},u_j]t^{-1}$ and $[v_n,v_j]=t[v_{n-1},v_j]t^{-1}$.
\end{proof}
\begin{theorem}The map $\xi_n:\NB_n\to\mathcal Q_n^\times$ given by 
$\xi_n(\sigma_i)=\frac{-1}{2q}(1+u_i+v_i+u_iv_i)$ and $\xi_n(\tau)=t$ defines a group homomorphism.
\end{theorem}

\begin{proof} Similarly as in the Gaussian case above, we use Lemma \ref{relationslemma} and \cite{quat} to reduce to checking $\xi_n(\tau\sigma_i\tau^{-1})=\xi_n(\sigma_{i+1})$ and $[\xi_n(\tau)]^{2n}=1$.  These are both immediate from the (last two) relations in $\mathcal{Q}_n$.
\end{proof}

Interestingly, $\QQ_n$ does not have an obvious local representation.  Instead, we obtain a $3$-local representation,  (see \cite[Theorem 5.28]{GHR}).  On the other hand, we can easily show that the image $\xi_n(\NB_n)$ is a finite group. First we show that,the conjugation action on the subalgebra $\hat{\mathcal{Q}}_n$ generated by $u_1,\ldots,u_n,v_1,\ldots,v_n$ is finite as follows.   Observe that $\hat{\QQ}_n$ is spanned by monomials of the form $$u_1^{\epsilon_1}\cdots u_n^{\epsilon_n}v_1^{\nu_1}\cdots v_n^{\nu_n}$$ where nonzero $\epsilon_i,\nu_i\in\{0,\pm 1\}$. The action of $\xi_n(\tau)=t$ obviously permutes this generating set. We can now compute, with $k=i\pm1$:
\begin{eqnarray*}
u_{i}\xi_n(\sigma_i) & = & \frac{-1}{2q}u_i(1+u_i+v_i+u_iv_i) = \frac{-1}{2q}(u_i-1+u_iv_i-v_i)\\
 & = & \frac{-1}{2q}(v_iu_iv_i+u_iv_iu_iv_i+u_iv_i+u_iu_iv_i) = \frac{-1}{2q}(v_i+u_iv_i+1+u_i)u_iv_i)\\ & = & \xi_n(\sigma_i)u_iv_i\end{eqnarray*} and
 \begin{eqnarray*}
v_i\xi_n(\sigma_i) & = & \frac{-1}{2q}v_i(1+u_i+v_i+u_iv_i)=\xi_n(\sigma_i)u_i\\
 & &\\
u_k\xi_n(\sigma_i) & = & \frac{-1}{2q}u_k(1+u_i+v_i+u_iv_i)= \xi_n(\sigma_i)u_kv_i\\
 & & \\
v_k\xi_n(\sigma_i) & = & \frac{-1}{2q}v_k(1+u_i+v_i+u_iv_i)=\xi_n(\sigma_i)(-u_iv_iv_k).
 \end{eqnarray*}

 Thus the conjugation action of $\xi_n(\NB_n)$ permutes a spanning set up to roots of unity so that $\xi_n(\NB_n)$ is finite modulo its center.

Now again, as in the Gaussian case we can see that the $\QQ_n$ is a finite dimensional semisimple algebra and the restriction to the center of $\xi_n(\NB_n)$ on any irreducible subrepresentation of the faithful regular representation gives a scalar of finite order, hence $\xi_n(\NB_n)$ has finite center and is thus a finite group.
%%%%%%%%%%%%%%%%%%%%%%%%%%%%%%%%%%%%%%%%%%%%%%%%%%%%%%%%%%%%%%%%%%%%%%%%%%%%%%%
%%%%%%%%%%%%%%%%%%%%%%%%%%%%%%%%%%%%%%%%%%%%%%%%%%%%%%%%%%%%%%%%%%%%%%%%%%%%%%%
\section{Extending low-dimensional $\B_n$ representations}
\label{ss:low-d}

There are gaps in the irreducible representation degrees of $\B_n$ (see \cite{LRagt} and references therein): for example $\B_n$ has no irreducible representations of dimension $2\leq d\leq n-3$, for $n\geq 5$.

\begin{lemma}
Let $(\rho,V)$ to be an irreducible $\NB_n$ representation, $w\in V$, and $0<\alpha$ be minimal such that $\tau^{-\alpha}w\in span\{\tau^{-\gamma}w\space|\space0\leq\gamma<\alpha\}$. Then $\tau^{-(\alpha+1)}w\in span\{\tau^{-\gamma}w\space|\space0\leq\gamma<\alpha\}$.
\end{lemma}
\begin{proof}
We have that $\tau^{-\alpha}w=\displaystyle\sum_{i=0}^{-\alpha+1}a_i\tau^{-i}w$ for $a_i$ scalars. Then
\begin{eqnarray*}
\tau^{-(\alpha+1)}w=\tau^{-1}(\tau^{-\alpha}w) & = & \tau^{-1}\left(\sum_{i=0}^{\alpha-1}a_i\tau^{-i}w\right) \\
 & = & \sum_{i=0}^{\alpha-1}a_i\tau^{-i-1}w\\
 & = & \left(\sum_{i=1}^{\alpha-1}a_{i-1}\tau^{-i}w\right)+\tau^{-\alpha}w
\end{eqnarray*}
which is in the span of $\{\tau^{-\gamma}w\space|\space0\leq\gamma<\alpha\}$.
\end{proof}\mbox{}

\begin{theorem}\label{thm:irrep-gaps}
Let $n\geq5$ and $(\rho,V)$ be an irreducible $\NB_n$ representation. If $\dim V= n-2$, then $\phi=\rho\vert_{\B_n}$ is also irreducible.
\end{theorem}
\begin{proof}
Assume that $(V,\rho)$ is an irreducible $\NB_n$ representation and to the contrary that $\rho\vert_{\B_n}$ is not irreducible. So we have that there exists a proper nonempty subspace $W$ (of minimal dimension, $1\leq q<n-2$) of $V$ such that $(\phi\vert_{W},W)$ is a $\B_n$ representation. Note that $W$ being minimal dimension guarantees that $(\phi\vert_{W},W)$ is irreducible. Since $n\geq5$, we have that the only irreducible representations of $\B_n$ of dimension strictly less than $n-2$ are 1-dimensional. Let $W$ be spanned by the vector $w\in V$. Since $W$ is a $\B_n$ invariant space, we have that $\sigma_iw=\lambda_iw$ for all $1\leq i\leq n-1$ (i.e. $w$ is an eigenvector for all $\rho(\sigma_i)$). Since the $\sigma_i$ are conjugate to each other, we have that $w$ is an eigenvector of the same eigenvalue. Since $\tau^{2n}=1$ and $V$ is irreducible, we also have that $\rho(\tau)^{\pm n}=\pm1$. Note that $M=\{w,\tau^{-1}w,\cdots,\tau^{-n}w\}$ is linearly dependent and that the span of $M$ is $\tau$ invariant. We may extract a basis $\beta=\{\tau^{-\alpha_1}w,\tau^{-\alpha_2}w,\dots,\tau^{-\alpha_k}w|0\leq\alpha_1<\alpha_2<\cdots<\alpha_k\leq n\}$ for $Q=span(M)$. Since $Q$ is $\tau$ invariant, we may instead use the basis $\beta^\prime=\tau^{\alpha_1}\beta$. This gives us that $w\in\beta^\prime$. From the above lemma, we obtain $\beta^\prime=\{w,\tau^{-1}w,\dots,\tau^{-\alpha}w\}$ (where $\alpha< n-2$) is a basis for $Q$. The restriction that $\alpha<n-2$ is from the fact that $Q$ is a subspace of $V$, and therefore $\dim Q\leq n-2$.\\ Let $0\leq\gamma\leq\alpha<n-2$. This means that $\sigma_{\gamma+1}\in \B_n$. From the relation\linebreak $ \tau^{\gamma}\sigma_1\tau^{-\gamma}=\sigma_{\gamma+1}$, we get that $\sigma_1\tau^{-\gamma}w=\tau^{-\gamma}\sigma_{\gamma+1}w=\lambda(\tau^\gamma w)$. This gives us that $Q$ is also $\sigma_1$ invariant. Hence $Q$ is both $\tau$ and $\sigma_1$ invariant, and therefore $\NB_n$ invariant. Since $V$ was irreducible, we have that $Q=V$. This means that $\beta^\prime$ is also a basis for $V$, and in this basis, $\rho(\sigma_1)=\lambda\cdot id_V$. Hence $$\rho(\sigma_2)=\rho(\tau)\rho(\sigma_1)\rho(\tau\iv)=\lambda\rho(\tau)\rho(\tau\iv)=\lambda\cdot id_V=\rho(\sigma_1).$$ This gives us that $\rho(\tau)\rho(\sigma_1)=\rho(\sigma_1)\rho(\tau)$. Which would give us a contradiction that $(\rho,V)$ is not an irreducible $\NB_n$ representation.
\end{proof}\mbox{}

We can now show that $\NB_n$ also has gaps in its irreducible representation degrees:
\begin{cor}
For $n\geq5$, the only irreducible representations of $\NB_n$ of dimension at most $n-3$ are 1-dimensional.
\end{cor}
\begin{proof}
Assume to the contrary that $(\rho,V)$ is an irreducible representation of $\NB_n$ with $2\leq\dim V<n-2$. Note that $\rho|_{\B_n}$ can not be irreducible, since $\B_n$ has no irreducible representations of dimension between 2 and $n-3$. Hence there exists a 1 dimensional subrepresentaion of $\B_n$. Following the proof of Theorem \ref{thm:irrep-gaps}, we would get that $\rho(\NB_n)$ is abelian, and therefore not irreducible.
\end{proof}\mbox{}

The following theorem is similar to the one above it, but it has the additional assumption that $\rho\vert_{\B_n}$ is completely reducible.
\begin{theorem}
Let $n\geq 5$ and $(\rho,V)$ be an irreducible $\NB_n$ representation. If $\dim V=n-1$ and $\rho|_{\B_n}$ is completely reducible, then $\rho|_{\B_n}$ is also irreducible.
\end{theorem}
\begin{proof}
Assume to the contrary that $\rho|_{\B_n}$ is completely reducible and not irreducible. Then we have two possibilities, $V=\bigoplus_{i=1}^{n-1}W_i$ or $V=W\oplus U$ where $W,W_i$ are all 1-dimensional subrepresentations, and $U$ is an n-2 irreducible subrepresentation of $V$ for $\B_n$. In either case, we have the existence of a 1-dimensional subrepresentation. From here, we follow the proof of Theorem \ref{thm:irrep-gaps}, and note that the inequality that $\alpha<n-2$ becomes $\alpha\leq n-2$. However, this still ensures that $Q$ is $\sigma_1$ invariant, because again, $\sigma_{n-1}=\tau^{n-2}\sigma_1\tau^{-n+2}$. So again, we would get the contradiction that $(\rho,V)$ is not an irreducible representation of $\NB_n$.
\end{proof}

%Below are examples of non-standard extensions of well known representations of $B_n$:\\
\begin{remark}Consider %the standard representation of $B_3$, which is given by $\displaystyle\sigma_1\mapsto\begin{pmatrix}0&z&0\\1&0&0\\0&0&1\end{pmatrix}$ and $\sigma_2\mapsto\begin{pmatrix}1&0&0\\0&0&z\\0&1&0\end{pmatrix}$ (where $z\neq1$ is a nonzero complex number). Letting $\tau\mapsto\begin{pmatrix}0&0&-t^2\\\frac{1}{t}&0&0\\0&\frac{1}{t}&0\end{pmatrix}$ gives a representation of $\NB_n$, and $\rho(\tau)\neq\lambda\rho(\sigma_1)\rho(\sigma_2)$. Here the only restriction is that if $t\neq1$, then $z\neq i$.\\
%
%Next, we have extensions of the Burau representation, for $n\geq3$. As a reminder,
the unreduced Burau representation, which is defined as follows: $\sigma_i\mapsto\left(\begin{array}{ccc}I_{i-1}&  & \\  &\begin{matrix}1-t&t\\1&0\end{matrix}& \\  & & I_{n-i-1}\end{array}\right)$.  It can easily be verified that the $(n\times 1)$ vector of all 1's is fixed by all of the $\sigma_i'$s. Consider the mapping $\tau\to\begin{pmatrix}0 & a^{n-1}\\ \frac{1}{a}\cdot I_{n-1}&0\end{pmatrix}$ where $I_{n-1}$ is the $n-1\times n-1$ identity. It can be checked that $\rho(\tau^{-1}\sigma_i\tau)=\rho(\sigma_{i+1})$ and $\rho(\tau)^{2n}=I_n$. Thus it gives us a representation of $\NB_n$. Note that any invariant subspace of $\rho(\NB_n)$ will also be invariant under $\rho(\B_n)$. It is known that the Burau representation is reducible with invariant subspaces of dimension $1$ and $n-1$. The 1 dimensional subrepresentation is spanned by the vector of 1's. The other is the subspace of $\C^n$ of all vectors whose entries add up to 0. If $a\neq1$, then we get that the vector of $1'$s is not fixed by $\tau$. If $a^n\neq1$, then we get that $\tau$ does not fix the $n-1$-dimensional subspace. This means that if $\NB_n$ has no invariant subspaces. Therefore if $a^n\neq1$, then the extension of the Burau representation described above is an irreducible representation of $\NB_n$ whose restriction to $\B_n$ is reducible.
As this shows, there exist irreducible representations of $\NB_n$ of dimension $n$ whose restriction to $\B_n$ is no longer irreducible.
\end{remark}

\subsection{Irreducible Representations of dimension 2}

From the fact that $\tau$ has order $~2n$, we may assume that we have chosen a basis for $V$ such that $\rho(\tau)=\begin{pmatrix}t_1&0\\0&t_2\end{pmatrix}$ where $~t_1,t_2$ are $2n^{th}$ roots of unity. As stated before, for all $n\geq5$, there are no irreducible 2 dimensional representations. This means we need only consider $n=2,3,$ and $4$. Since we are wanting irreducible reps, we have that $t_1\neq t_2$. Similarly we have that $\rho(\sigma_1)$ is not upper or lower triangular, as otherwise $(1,0)$ or $(0,1)$ would generate an invariant subspace. Due to rescaling, we may assume that $\rho(\sigma_1)=\begin{pmatrix}a&1\\c&d\end{pmatrix}$. Since we do not want diagonal or triangular matrices, this means that $c\neq0$. Wanting our representations to be irreducible, we also have at least one of $a$ or $d$ are nonzero.
\begin{prop} Any irreducible dimension 2 representation of $\NB_2,\NB_3$ or $\NB_4$ is isomorphic to one of those forms in Table \ref{reptable}.\end{prop}
\begin{table}
$\begin{array}{|c|c|c|}\hline
\rho(\sigma_1) & \rho(\tau) & \text{restrictions} \\
 & & \\\hline\hline
 & & \\
\begin{pmatrix}a&1\\a^2-ad+d^2&d\end{pmatrix} & \begin{pmatrix}-t_2&0\\0&t_2\end{pmatrix} & a\neq d \\
 & &\\\hline
 & & \\
\begin{pmatrix}a&1\\-a^2+ad-d^2&d\end{pmatrix} & \pm\begin{pmatrix}1&0\\0&i\end{pmatrix},  \pm\begin{pmatrix}i&0\\0&1\end{pmatrix}& a\neq d, n=2\\
 
 & & \\\hline
 & & \\
\begin{pmatrix}a&1\\\frac{-1}{2}(a^2-ad+d^2)&d\end{pmatrix} & \pm\begin{pmatrix}e^{\pm i \frac{\pi}{3}}&0\\0&1\end{pmatrix}, \pm\begin{pmatrix}1&0\\0&e^{\pm i\frac{\pi}{3}}\end{pmatrix}, & a\neq d, n=3 \\
 & \pm\begin{pmatrix}e^{i\frac{2\pi}{3}}&0\\0&e^{i\frac{\pi}{3}}\end{pmatrix}, \pm\begin{pmatrix}e^{i\frac{\pi}{3}}&0\\0&e^{i\frac{2\pi}{3}}\end{pmatrix}, & \\
 & & \\\hline
 & &  \\
\begin{pmatrix}\omega d&1\\ c &d\end{pmatrix} & \pm\begin{pmatrix}1&0\\0&e^{\pm i\frac{2\pi}{3}}\end{pmatrix}, \pm\begin{pmatrix}e^{\pm i\frac{2\pi}{3}}&0\\0&e^{\mp i\frac{2\pi}{3}}\end{pmatrix} & d\neq0, c\neq\omega d^2, n=3\\
&&\\
\hline\end{array}$ \caption{Dimension $2$ Representations of $\NB_n$, $2\leq n\leq 4$.  Here $\omega$ is a primitive $3$rd root of unity. \label{reptable}}
\end{table}

%\section{Sandpit (temporary)}
%\newcommand{\ppm}[1]{{\textcolor{green}{#1}}}
%\newcommand{\er}[1]{{\textcolor{blue}{#1}}}
%ppm: I just try to get the hang of things down here, 
%before messing with the real stuff. 
%\ppm{If it works, I will later put some questions in green %perhaps...
%Does this work as a discussion protocol? Or am I missing something %better?}
%\er{Seems ok to me} \ppm{  :-)  }
\iffalse
\subsection{What factors through affine Hecke?}
...
\er{The paper:
Communications in Contemporary Mathematics
Vol. 19, No. 3 (2017) might be useful in this regard.}
\fi
%\medskip
%\input loop   %% does this work in Overleaf? LOL. Yes.
\section{Representations from Topological Physics}
The category-theoretic approach to TQFTs as well as its statistical-mechanical predecessor suggest that similar techniques should yield physically relevant necklace braid group representations.  Here we discuss some related ideas from braided fusion categories and spin chain models that to illustrate this direction.

\subsection{Representations from Braided Fusion Categories}
\label{ss:BFC}

From any (unitary) braided fusion category $\CC$ Walker and Wang \cite{ww12} construct a $(3+1)$-TQFT.  In the extreme case of modular $\CC$ their construction is degenerate.  In the other extreme case of a symmetric braided fusion category of the form $\Rep(G)$ for a finite group $G$ one recovers the Dijkgraaf-Witten theory \cite{DijkgraafPasquierRoche}.  The most interesting case is neither symmetric nor modular.  As we expect a $(3+1)$TQFT to provide representations of motion groups of links in $\R^3$ or $S^3$, we briefly explore this possibility directly from the categorical perspective.

Given a braided fusion category $\CC$ we obtain 
from any object $X$ 
a representation of $\B_n$ on $\End(X^{\otimes n})$, via the braiding $c_{X,X}$
(we use the notation of \cite{RW}, and assume $\CC$ is strictly associative for notational convenience).  More specifically, the map $\sigma_i\mapsto \one_X^{\otimes i-1}\otimes c_{X,X}\otimes \one^{\otimes n-i-1}$ defines a group homomorphism $\B_n\rightarrow \Aut(X^{\otimes n})$, which we may regard as a \emph{categorical} representation of $\B_n$.  If $\{X_i\}$ is a complete set of representatives of the isomorphism classes then $\Aut(X^{\otimes n})$ acts faithfully on $\mathcal{H}_{X,n}:=\bigoplus_i\Hom(X_i,X^{\otimes n})$ by composition, which then gives us an honest linear representation $(\rho_X,\mathcal{H}_{X,n})$ of $\B_n$.  Are there categorical representations of $\NB_n$?  That is, group homomorphisms $\NB_n\rightarrow\Aut(Y(n))$ for some objects $Y(n)$ depending on $n$?
As the representations $(\rho_X,\mathcal{H}_{X,n})$ are completely reducible we may extend them to $\NB_n$ in the standard way as in Theorem \ref{stdexn} by rescaling the image of $\gamma$.    
It would be interesting to find categorical non-standard extensions. 

One cannot expect categorical $\B_n$ representations to admit extensions to the loop braid group $\LB_n$.  Of course if the braiding $c_{X,X}$ is symmetric: $c_{X,X}^2=\one_{X\otimes X}$ then we can extend $(\rho_X,\mathcal{H}_{X,n})$ to $\LB_n$ by $s_i\rightarrow \sigma_i$.  However, the map $\sigma_i\mapsto s_i$ savagely reduces $\LB_n$ to the symmetric group $S_n$, so this is not an interesting choice.  
\iffalse
However, an intriguing possibility is to consider an object $Z\in\CC$ such that $\End(Z^{\otimes 2})$ admits two maps $\sigma_{Z,Z}$ and $s_{Z,Z}$ such that $\sigma_i\mapsto \one_Z^{\otimes i-1}\otimes \sigma_{Z,Z}\otimes \one_Z^{\otimes n-i-1}$ and $s_i\mapsto \one_Z^{\otimes i-1}\otimes s_{Z,Z}\otimes \one_Z^{\otimes n-i-1}$ give a representation of $\LB_n$.  For example, one might consider an object $Z=X\otimes Y$ where $c_{X,Y}c_{Y,X}=\one_{Y\otimes X}$ and $c_{Y,Y}^2=\one_{Y\otimes Y}$, and construct $\sigma_{Z,Z}$ and $s_{Z,Z}$ as combinations of $c_{X,X},c_{X,Y}$ and $c_{Y,Y}$.  Preliminary attempts (cabling, conjugating $c_{X,X}$ and $c_{Y,Y}$ by $c_{X,Y}$, etc.) do not seem to work.
\ppm{[paul: I don't understand this yet]}

\ppm{this question perhaps factors partly through the loop braid connection?
it happens sometimes, of course, that the natural symmetric gp rep in
$\End(X^{\otimes n})$ plays nicely with the $\B_n$ rep.
- meaning that we get a loop braid rep.
for example this works for the double construction. but not in general.
And when it happens then the cyclic perm built from the symm gp part
will do the job... by the Lemma/theorem\ref{lem:N2L}.
Of course you are asking for less (or anyway different) than this. 
but somehow the spirit
of the point is the same...} 
\fi
 The annular braid group $\CB_n$ does admit such representations as follows (c.f. \cite{OW}). Let $X,Y$ be objects in a braided fusion category, and set $$R_i=I_Y\ot I_X^{\ot i-1}\ot R_{X,X}\ot I_X^{\ot n-i-1}\in\End(Y\ot X^{\ot n})$$ for $1\leq i\leq n-1$, $\alpha_{Y,X}=(R_{X,Y}R_{Y,X}\ot I_X^{\ot n-1})$ and $T_{Y,X}=\alpha_{Y,X}R_1\cdots R_{n-1}.$ Then $\sigma_i\rightarrow R_i$ and $\tau\rightarrow T_{Y,X}$ satisfy (B1), (B2) and (N1) after defining $R_n=T_{Y,X}R_{n-1}T_{Y,X}^{-1}$.  Notice that (B1) and (B2) are immediate for small $i$ but we must check that $R_{n-1}R_nR_{n-1}=R_nR_{n-1}R_n$ and $R_1R_nR_1=R_nR_1R_n$ and the far-commutation for $R_n$ with $R_2,\ldots,R_{n-2}$.  Observing that $\alpha_{Y,X}$ commutes with $R_j$ for $2\leq j\leq n$, relation (N1) can be seen in the following way:
\begin{eqnarray*} T_{Y,X}R_i&=&\\\alpha_{Y,X} R_1\cdots R_{i-1}R_iR_{i+1}R_i\cdots R_{n-1}&=&
\alpha_{Y,X} R_1\cdots R_{i-1}R_{i+1}R_{i}R_{i+1}\cdots R_{n-1}\\ =\alpha_{Y,X}R_{i+1} R_1\cdots R_{i-1}R_{i}R_{i+1}\cdots R_{n-1}&=&
R_{i+1}\alpha_{Y,X}R_1\cdots R_{n-1}\\&=&R_{i+1}T_{Y,X}.
\end{eqnarray*}
For such a representation to factor over $\NB_n$ one needs $(T_{Y,X})^{2n}=\one_{Y\ot X^{\ot n}}$.  Of course if this is required for all $n$ then $T_{Y,X}^2=\one_{Y,X}$--which is precisely the condition that $X$ and $Y$ centralise each other in the sense of M\"uger \cite{Mug}.  If $Y$ is transparent \cite{Brug} then this condition is satisfied for all $X$.  Indeed, in \cite{CTW} a 3 loop configuration with fermion interactions is considered, which would be a particular case of this set up.

\begin{example} Consider the Ising braided fusion category with simple objects $\unit$, $\psi$ and $\sigma$.  Setting $Y=\psi$ (the Majorana fermion) and $X=\sigma$ (the Ising anyon) we find that $c_{X,Y}c_{Y,X}=-\one_{Y\ot X}$ where $Y\ot X\cong\sigma$ so that $\one_{Y\ot X}\in \End(\sigma)\cong\C$.  In particular $\alpha_{Y,X}=-\one_{Y\ot X^{\ot n}}$, and we can easily adjust the scalar to ensure $(T_{Y,X})^{2n}=\one$. Since $Y\ot X^{\ot n}\cong X^{\ot n}$ this does not provide a very interesting example.
\end{example}

\medskip

\subsection{Necklaces and spin chains}\label{ss:XXZ}

Beautiful classes of representations of $\B_n$ are constructed in statistical mechanics, both from open spin chains and from transfer matrices. 
Drawn suitably, a necklace has some similarity with a `thickened' 
{\em periodic} 
spin chain.
\[
%linear pic
\includegraphics[width=4in]{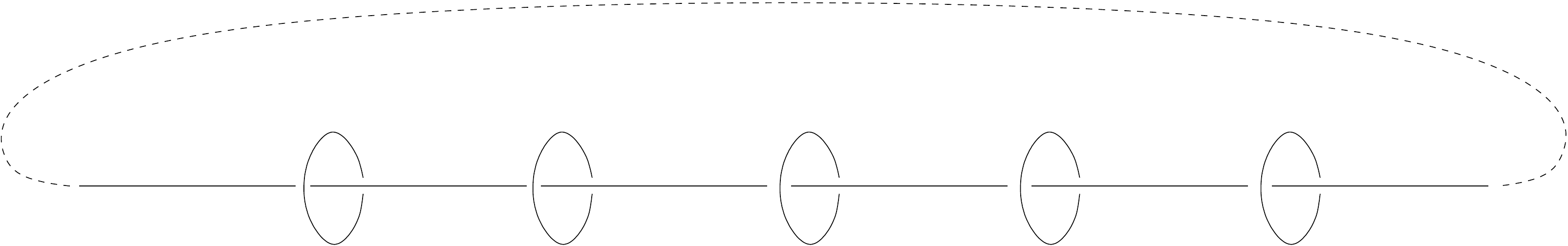}
\]
The main superficial difference with $\B_n$ 
is that there is no 
algebraic way to make this chain open. 
This in turn corresponds to the fact that, unlike the natural inclusion of $\B_{n-1}$ in $\B_n$, and similarly for $\LB_n$, 
$\NB_{n-1}$ does not include in an analogous way in $\NB_n$.

Consider for example the $n$-site XXZ spin chain \cite{LiebMattis}. 
Here we associate the leapfrog motion $\sigma_i$ to the local spin-spin interaction and, formally, the motion $\tau$ to the periodic symmetry of the chain. Naively this gives $\tau^n=1$. But the physical system may be endowed with what is sometimes called a cohomology seam,
generally localised in a boundary condition, (cf. e.g. \cite{BaxterTemperleyAshley,MartinSaleur}) so that the image of $\tau$ may generate groups of various orders.

Let us start by using this construction to discuss Theorem~\ref{stdexn} a little more. 
%Already here the case $n=2$ is not quite trivial. 
We proceed as follows. 
Firstly, a neat way to construct the spin chain itself is from 
an XXZ TQFT (cf. e.g. 
\cite{LiebMattis,Jimbo,Drinfeld,Martin91,wang10}
--- the `bordisms' here are plane-embedded 1-manifolds) 
with parameter $q=t^2$. 
This system is monoidally generated by 
\newcommand{\uu}{{\mathsf u}}
\[
\uu = 
\raisebox{-.0651in}{
\includegraphics[width=1.2991cm]{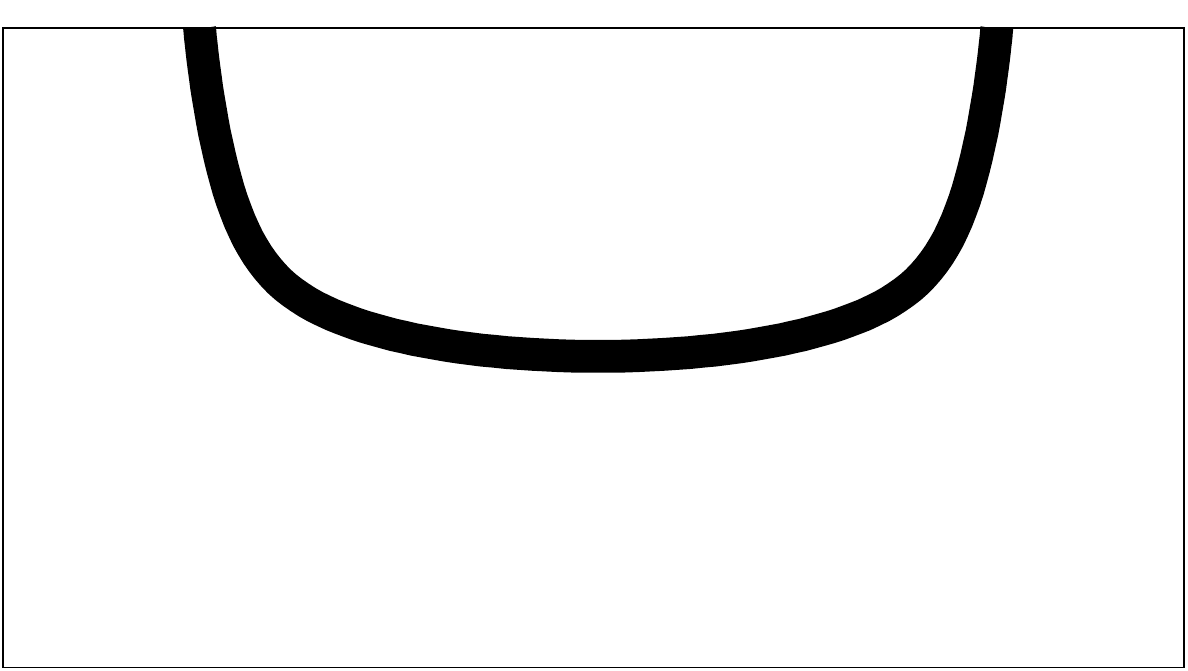}
}
\mapsto (0,-t,t^{-1}, 0)
\hspace{.4321in}
\mbox{and}
\hspace{.4321in}
1_1 = 
\raisebox{-.07651in}{
\includegraphics[width=.76591cm]{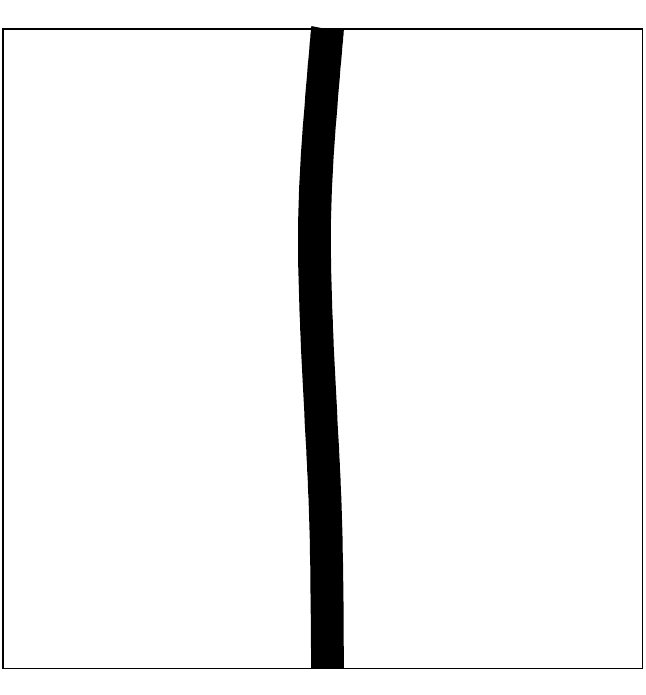}
}
\mapsto \mat{cc} 1&0 \\ 0&1 \tam
\]
Setting $U = \uu^t \uu$ (hereafter we will simply identify elements with their images in the category $Vect$ as above,
so that $\uu^t$ means transpose); 
and, for $n$ given, 
setting $U_i$ to be $U$ localised in the $i$-th position in an $n$-fold  tensor product in the usual way, we recall 
firstly that 
\[
H = \sum_{i=1}^{n-1} U_i
\]
is the `open' XXZ Hamiltonian,
and also that 
\beq \label{eq:TLbraid}
g_i = 1 - q U_i
\eq
obeys the braid relation. 
Here we write $\rho$ for this representation
(notationally suppressing the dependencies on $q$ and $n$). 

For $n=2$ there is nothing to check here, and yet it is not quite trivial viewed in the context of Theorem~\ref{stdexn}.
Factoring through \eqref{eq:TLbraid} and $\rho$,
the image of 
$\gamma^{n=2}$  % = \gamma^2 $ %\tau^n$ 
is simply 
\[
\rho( g_1^2 ) = \mat{cccc} 
1 \\
& 1-q^2 +q^4 & -q(1-q^2 ) \\
&-q(1-q^2)&q^2& \\
&&&1 \tam
\]
with eigenvalues $1,q^4$.
If $q^4 \neq 1$ or $q^2 =1$ then the Theorem applies directly
(and gives a very uninteresting representation).
If $q =i$, however, the Theorem does not apply directly
(all eigenvalues are 1 but the matrix is not the identity matrix, so the representation 
of $\B_n$
is not completely reducible
--- indeed it contains a representation isomorphic to that in Remark~\ref{rem:indec}, for $n=2$, as a direct summand).
(The case of $q=i$ exhibits non-complete-reducibility
%is true 
for general even $n$. We start with $n=2$ to postpone
unhelpfully non-trivial algebra.)
We note for future reference that the `uninteresting' representation with $q^4 = -1$ (to which the Theorem formally does apply) obeys 
$\rho(\tau^{2n}  )  = g_1^4 = 1$ 
{\em without need of any renormalisation}. 
Let us call this eventuality 
($D=1$ in  Theorem~\ref{stdexn}) a 
{\em flat} standard extension.

In general the problem of computing the spectrum of $\gamma^n$ here is a kind of elementary integrable system 
(see e.g. \cite{Martin91}). By centrality it acts by a scalar on each indecomposable summand of tensor space.
These summands are indexed by the integers $l$ congruent to $n$ mod.2 (the XXZ charge or  number of propagating lines
as 
for example in \cite{Baxter}), so the complete spectrum is given by
%\[ table \]
%
\[
\xymatrix@R=2pt{
n\setminus l
   & 0    & 1    & 2     & 3      & 4      & 5    & 6  \\
 0 & 1                                      \\
 1 &      & 1                               \\
 2 & q^4  &        & 1                       \\
 3 &       & q^6    &      & 1                   \\
 4 & q^{12} &        & q^8   &        & 1          \\
 5 &       & q^{16}  &       & q^{10}  &       &  1  \\
 6 & q^{24} &        & q^{20} &        & q^{12} &   & 1 \\
 7 &       & q^{30}  &       & q^{24}  &       &q^{14} &&1
}
\]
(the general pattern for larger $n$ will be clear). 
A necessary condition for $\rho( \gamma^{2n} ) =1$ is,
of course, that 
%the spectrum of $\rho( \gamma^n )$ lies in $\{ 1,-1 \}$.
the only eigenvalue is 1.
We see that for this to be true for all $n$ 
(indeed for it to be true for $n=5$) we require 
$q^4 = 1$. 
(And to be true for all even $n$ we require $q^8 =1$.)
The case $q^2 =1$ factors through the symmetric group, so the more interesting case is $q^2 = -1$. 

\begin{prop}
Setting $q^2 =-1$ then 
(a) for all odd $n$,
$\rho$ gives a representation of $\NB_n$ via $\rho(\tau)=\rho(\gamma)$ (i.e. with trivial $D$). 
%for all odd $n$. 
Furthermore $\rho(\tau^n) \neq 1$.
(b) for even $n$ we never get a representation this way,
i.e. we do not get a flat standard extension.
\end{prop}
\begin{proof}
(a) For odd $n$ these representations of $\B_n$ are completely reducible by
\cite[\S7.3 Th.2]{Martin91}, and every eigenvalue of $\rho(\gamma^{2n})$ takes the form $q^{4m}$ for some $m$.
Meanwhile the spectrum of $\rho(\gamma^n)$ always contains 
1 and $q^{2m}$ with $m$ odd.
\\
(b) For even $n$ these representations are not completely reducible (they are faithful on a non-semisimple quotient algebra). It remains to show in particular that $\rho(\gamma^{2n})$ has a non-trivial Jordan form.
To see this note that $j=U^{\otimes n/2}$ is both central 
and radical in the quotient algebra; that $\gamma$
expressed in the bordism basis contains  $j$ with non-vanishing coefficient; and that this holds also for any power of $\gamma$.
\end{proof}

%\noindent
%\ppm{paul, quick note to self: use as mini example generator for Th.2.2? And try one seam example?}
%
%For the case $n=3$ we have ...

The difference between odd and even cases is well-known in the XXZ setting, but intriguing here. It again suggests to use a cohomology seam  in the manner of
\cite{BaxterTemperleyAshley,MartinSaleur}.

Closed boundary conditions have been studied extensively
in the XXZ and indeed the wider spin chain setting
(see e.g. \cite{McKay} for recent references). 
The seam approach extends the 
TQFT by an operator that acts only on the first position
--- represented in `bordisms' by a blob
\cite{MartinSaleur,Doikou}.
A simple example of this is
\newcommand{\ff}{{\mathsf f}} 
\[
\ff = %\ppm{pic} 
\raisebox{-.1in}{
\includegraphics[width=.3in]{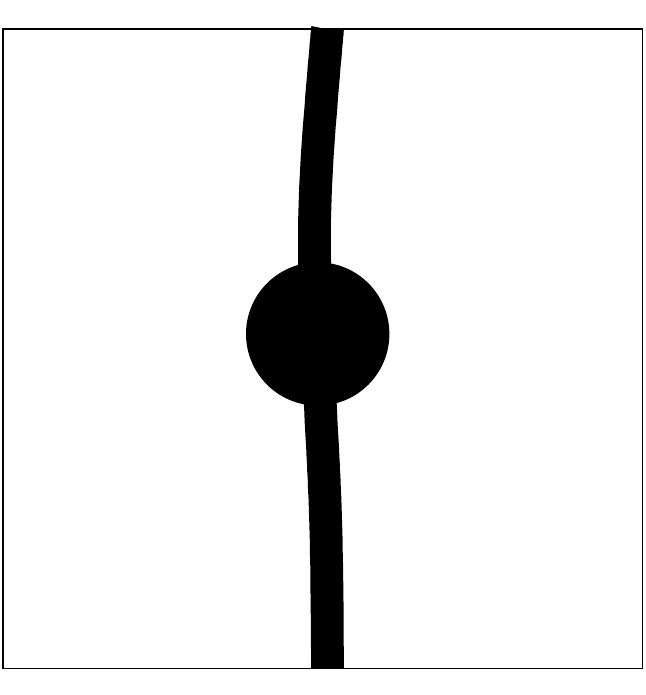}}
\mapsto 
\frac{1}{a+a^{-1}} \mat{cc} a & 1 \\ 1&a^{-1} \tam
\]
(always localised in the first position in the tensor product).
Such an extension introduces a new `boundary' parameter into the chain, normally given \cite{MartinSaleur} 
by the value of the 
`topological loop' scalar
\beq
y_\ff = \uu (\ff \otimes 1_1 ) \uu^t
=
\raisebox{-.2in}{
\includegraphics[width=.5in]{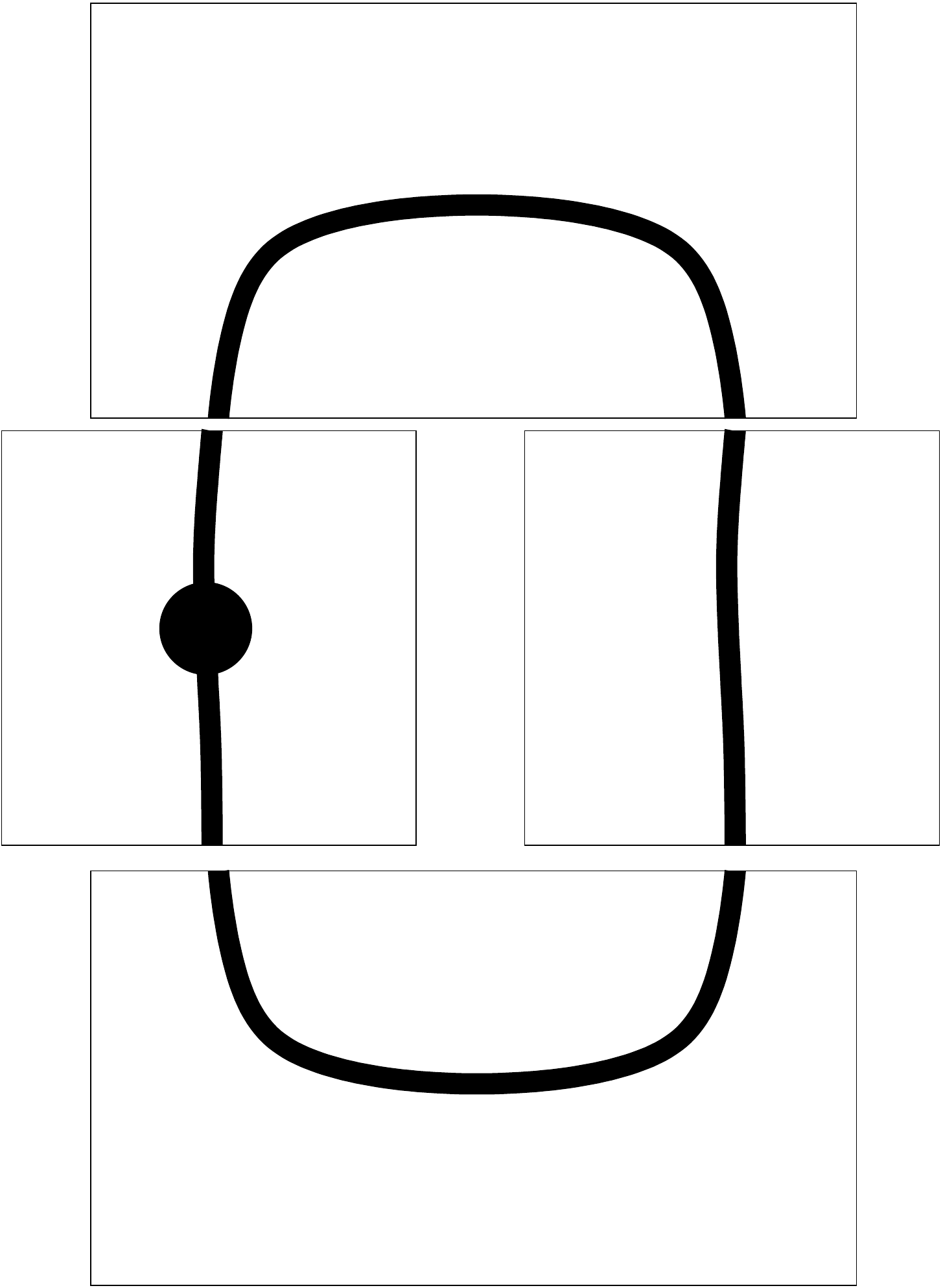}}={\frac {{a}{t}^{2}+(at^2)^{-1}}{ \left( {a}+a^{-1} \right) }}
\eq
%\ppm{paul, get the transposes the right way round, here and above!}
(N.B. bottom-to-top stacking convention).
Such an extension then allows the construction of a %generalised 
`seamed' braid translation operator
\newcommand{\BB}{\beta}
\beq \label{eq:seam}
\BB = (1+x\ff ) \gamma 
\eq
where $x=\frac{q - q^{-1}}{q^{-1} - y_\ff}$.
Just as with $\gamma$ itself in \eqref{eq:sigman} %braidtrans} 
we define 
$ g_0 =  \BB g_{n-1} \BB^{-1}$
whereupon
\beq \label{eq:bgb}
\BB g_i \BB^{-1} = g_{i+1}
\eq
with indices understood periodically
\cite{MartinSaleur}.
Indeed this holds not only in $\rho$ but in the setting
of the algebra $b_n(q,a)$ of abstract generators
$g_i $ and $\ff$.
Thus we have the following.

\begin{theorem}
Applying 
the $D$-matrix method of 
Theorem~\ref{stdexn} to the generators $g_i$ and the braid translator $\BB$ we obtain a two-parameter representation of $\NB_n$ whenever the 
corresponding complete reducibility condition is satisfied. The reducibility condition is satisfied, for example, when the parameters $q,a$ are indeterminate --- i.e.  %at least 
on a Zariski open subset of parameter space.
\end{theorem}

\begin{proof}
By \eqref{eq:bgb} $\zeta\mapsto\BB$ gives a representation of $\CB_n$.
Since $\zeta^n$ is central in $\CB_n$ the proof of
Theorem~\ref{stdexn} generalises.
The Zariski open property follows from \cite{MartinSaleur}.
%\ppm{finish me - main claim argument should follow similar to 
%previous thm. and centrality. 
%second claim from rep theory of blob.}
\end{proof}

The system required to determine the eigenvalues of $\BB^n$,
and hence determine the $D$ matrix explicitly,
is substantially more involved than the $\gamma$ case above.
We will discuss this elsewhere.
%\ppm{Paul, comment briefly on 
%generalised Lemma~\ref{lem:1} here wrt unitarity/continuity!}

%... Jordan form ...

%...

\end{document}